\documentclass[10pt, a4paper]{amsart}

\usepackage[T1]{fontenc}
\usepackage{microtype}

\usepackage{amsmath}								
\usepackage{amssymb}
\usepackage{amsthm}
\usepackage{amscd}
\usepackage{amsfonts}
\usepackage{stmaryrd}

\usepackage{extarrows}

\usepackage[bookmarks=true, colorlinks=true, linkcolor=blue!50!black,
citecolor=orange!50!black, urlcolor=orange!50!black, pdfencoding=unicode]{hyperref}
\usepackage{color}

\usepackage{tikz}									
\usetikzlibrary{matrix}
\usetikzlibrary{patterns}
\usetikzlibrary{matrix}
\usetikzlibrary{positioning}
\usetikzlibrary{decorations.pathmorphing}
\usetikzlibrary{cd}

\usepackage{nicefrac}
\usepackage{fullpage}
\usepackage{enumerate}

\newtheorem*{theorem*}{Theorem}

\newtheorem{maintheorem}{Theorem}[section]

\newtheorem{theorem}{Theorem}[section]
\newtheorem{lemma}[theorem]{Lemma}
\newtheorem{proposition}[theorem]{Proposition}
\newtheorem{corollary}[theorem]{Corollary} 

\theoremstyle{definition}
\newtheorem{definition}[theorem]{Definition}
\newtheorem{example}[theorem]{Example}

\newtheorem{remark}[theorem]{Remark}

\newtheoremstyle{myitemstyle}						
	{}			
	{}			
	{}			
	{}			
	{}			
	{.}			
	{ }			
	{}			
\theoremstyle{myitemstyle}
\newtheorem{myitemthm}{}

\newcommand{\R}{\mathbb{R}}

\newcommand{\Z}{\mathbb{Z}}
\newcommand{\ZZ}{\mathbb{Z}}

\renewcommand{\C}{\mathbb{C}}

\newcommand{\PP}{\mathbb{P}}

\newcommand{\RR}{\mathbb{R}}
\newcommand{\A}{\mathbb{A}}
\renewcommand{\G}{\mathbb{G}}

\newcommand{\FF}{\mathbb{F}}

\newcommand{\TT}{\mathbb{T}}

\newcommand{\calE}{\mathcal{E}}

\newcommand{\calH}{\mathcal{H}}

\newcommand{\calL}{\mathcal{L}}
\newcommand{\calM}{\mathcal{M}}

\newcommand{\calO}{\mathcal{O}}

\newcommand{\calU}{\mathcal{U}}

\newcommand{\calX}{\mathcal{X}}

\newcommand{\mc}[1]{{\mathcal{#1}}}

\DeclareMathOperator{\Pic}{Pic}

\DeclareMathOperator{\Hom}{Hom}

\DeclareMathOperator{\Aut}{Aut}

\DeclareMathOperator{\val}{val}
\DeclareMathOperator{\characteristic}{char}

\DeclareMathOperator{\Trop}{Trop}

\DeclareMathOperator{\id}{id}

\DeclareMathOperator{\GL}{GL}

\DeclareMathOperator{\PGL}{PGL}
\DeclareMathOperator{\Div}{Div}

\DeclareMathOperator{\Sym}{Sym}

\DeclareMathOperator{\SL}{SL}

\DeclareMathOperator{\Alb}{Alb}
\DeclareMathOperator{\Jac}{Jac}
\DeclareMathOperator{\PDiv}{PDiv}
\DeclareMathOperator{\Bun}{Bun}

\DeclareMathOperator{\ord}{ord}

\DeclareMathOperator{\PL}{\mathcal{PL}}
\DeclareMathOperator{\ch}{ch}
\DeclareMathOperator{\td}{td}
\DeclareMathOperator{\gr}{gr}
\DeclareMathOperator{\rk}{rk}
\DeclareMathOperator{\calSpec}{\mathcal{S}pec}

\let\Rat\relax
\DeclareMathOperator{\Rat}{Rat}

\let\div\relax
\DeclareMathOperator{\div}{div}

\newcommand{\an}{\mathrm{an}}

\newcommand{\ind}{\mathrm{ind}}
\newcommand{\trop}{{\mathrm{trop}}}

\newcommand{\free}{{\mathrm{free}}}

\newcommand{\Ga}{\Gamma}
\newcommand{\tGa}{\widetilde{\Gamma}}
\newcommand{\tG}{\widetilde{G}}
\newcommand{\te}{\widetilde{e}}
\newcommand{\tp}{\widetilde{p}}
\newcommand{\tv}{\widetilde{v}}

\title{Principal bundles on metric graphs: the $\GL_n$ case} 

\date{}

\author{Andreas Gross}
\address{Institut f\"ur Mathematik, Goethe--Universit\"at Frankfurt,
60325 Frankfurt am Main, Germany}
\email{gross@math.uni-frankfurt.de}

\author{Martin Ulirsch}
\address{Institut f\"ur Mathematik, Goethe--Universit\"at Frankfurt,
60325 Frankfurt am Main, Germany}
\email{ulirsch@math.uni-frankfurt.de}

\author{Dmitry Zakharov}
\address{Department of Mathematics, Central Michigan University, Mount Pleasant, MI 48859, USA}
\email{dvzakharov@gmail.com}

\begin{document}

\maketitle

\begin{abstract} Using the notion of a root datum of a reductive group $G$ we propose a tropical analogue of a principal $G$-bundle on a metric graph. We focus on the case $G=\GL_n$, i.e.\ the case of vector bundles. Here we give a characterization of vector bundles in terms of multidivisors and use this description to prove analogues of the Weil--Riemann--Roch theorem and the Narasimhan--Seshadri correspondence. We proceed by studying the process of tropicalization. In particular, we show that the non-Archimedean skeleton of the moduli space of semistable vector bundles on a Tate curve is isomorphic to a certain component of the moduli space of semistable tropical vector bundles on its dual metric graph. 
\end{abstract}

\setcounter{tocdepth}{1}
\tableofcontents


\section*{Introduction}

Since its introduction to the mathematical world in~\cite{MikhalkinZharkov}, the theory of divisors and linear systems on metric graphs has stimulated a multitude of further work during the last decade. It naturally aligned with the simultaneously developed theory of divisors, linear systems, and chip-firing on finite graphs introduced in~\cite{BakerNorine}, whose first highlight is the Riemann--Roch theorem in~\cite{BakerNorine} (also see~\cite{MikhalkinZharkov, GathmannKerber} for its generalization to metric graphs). Using Baker's method of specialization~\cite{Baker_specialization}, new applications to the classical subject of Brill--Noether theory on an algebraic curve were found (see~\cite{BakerJensen} for an overview). We mention the groundbreaking paper~\cite{CDPR}, which includes a new proof of the Brill--Noether theorem and point to ~\cite{Pflueger, JensenRanganathan, CookPowellJensen_gonality, CookPowellJensen_HurwitzBrillNoether, LenUlirsch} for (Hurwitz--)Brill--Noether theorems with gonality conditions (also see~\cite{Larson_refinedBrillNoether} for a non-tropical approach) and~\cite{FarkasJensenPayne_Kodaira, FarkasJensenPayne_nonabelian} for applications to the birational geometry of moduli spaces, which both take us significantly beyond what was known classically. 

What all of these works have in common is that they always focus on the case of divisors or equivalently of line bundles, i.e.\ the (very much abelian) rank one situation. In this article we endeavour towards a (non-abelian) higher rank generalization and use the notion of a root datum (coming from the classification of reductive groups) to propose a tropical analogue of principal $G$-bundles for a reductive group $G$. 

Let $\Gamma$ be a metric graph and let $\calH_\Gamma$ be the sheaf of real-valued harmonic functions with integer slopes on $\Gamma$. When $G=\GL_n$, our definition amounts to considering $S_n\ltimes\calH_\Gamma^n$-torsors on $\Gamma$, an object that we refer to as a \emph{vector bundle of rank} $n$ on $\Gamma$. This definition agrees with the geometric version of tropical vector bundles as defined in~\cite{TropChern}. In Section~\ref{section_multidivisors} we show that any vector bundle of rank $n$ on $\Gamma$ is the pushforward of a line bundle along a suitable free cover of $\Ga$ of degree $n$, and use this observation to describe isomorphism classes vector bundles on $\Gamma$ as linear equivalence classes of so-called \emph{multidivisors} on $\Gamma$, a notion inspired by Weil's matrix divisors coming from~\cite{Weil}. This description, for example, allows for the classification of vector bundles on metric graphs of genus zero in analogy with the Birkhoff--Grothendieck theorem~\cite{Grothendieck_vectorbundlesonP1} (see Example~\ref{example_BirkhoffGrothendieck} below) and of genus one in analogy with Atiyah's classication in ~\cite{Atiyah_ellipticvectorbundles} (see Section~\ref{section_Atiyah} below).

\subsection*{A tropical Weil--Riemann--Roch theorem} In~\cite{Weil}, Weil generalized the classical Riemann--Roch theorem to the case of vector bundles on a compact Riemann surface. From a modern perspective, this is now an immediate corollary of the Hirzebruch--Riemann--Roch theorem. Given a vector bundle $E$ on a compact Riemann surface $X$ of genus $g$, the Hirzebruch--Riemann--Roch formula states that
\begin{equation*}
    h^0(X,E)-h^1(X,E)=\int_X \ch(E)\cdot \td(X) \ .
\end{equation*}
The Chern character of $E$ and the Todd class of $X$ are given respectively by
\begin{equation*}
\ch(E)=c_0(E)+c_1(E)=\rk(E)+\deg(E)\cdot[X]
\end{equation*}
and 
\begin{equation*}\td(X)=c_0(T_X)+\frac{1}{2}c_1(T_X)=1-(g-1)\cdot [X],
\end{equation*}
while $h^1(X,E)=h^0(X,E^\ast\otimes\omega_X)$ by Serre duality. Hence we integrate over $X$ to obtain
\begin{equation}\label{eq_WeilRiemannRochclassical}
    h^0(X,E)-h^0(X,E^\ast\otimes\omega_X)=\deg(E) - \rk(E) \cdot (g-1) \ .
\end{equation}

In Section~\ref{section_WeilRiemannRoch} we use multidivisors to define a generalization of the \emph{Baker-Norine rank} $r_\Gamma$ (originally introduced in \cite{BakerNorine}) to vector bundles  on a compact metric graph and prove an analogue of \eqref{eq_WeilRiemannRochclassical}. 

\begin{maintheorem}[Weil--Riemann--Roch theorem, Theorem \ref{thm_WeilRiemannRoch}]\label{mainthm_WeilRiemannRoch}
Let $E$ be a vector bundle of rank $n$ on a compact and connected metric graph $\Gamma$ of genus $g$. Writing $\omega_\Gamma$ for the line bundle associated to the canonical divisor $K_\Gamma$, we have
\begin{equation*}
    r_\Gamma(E)-r_\Gamma(E^\ast\otimes\omega_\Gamma)=\deg(E)-n\cdot (g-1) \ .
\end{equation*}
\end{maintheorem}

The proof of Theorem~\ref{mainthm_WeilRiemannRoch} uses multidivisors and a special case of the Riemann--Hurwitz formula to deduce the statement from the rank one Riemann--Roch theorem of~\cite{BakerNorine,MikhalkinZharkov, GathmannKerber} for possibly disconnected metric graphs.

\subsection*{Semistable bundles and a Narasimhan--Seshadri correspondence} The Abel--Jacobi theorem on a compact Riemann surface $X$ provides a natural correspondence between topology (the Jacobian) and algebraic geometry (the Picard group $\Pic_0(X))$. In order to generalize this correspondence to the higher rank case, Narasimhan and Seshadri~\cite{NarasimhanSeshadri}, building on ideas of Weil~\cite{Weil}, prove that there is a natural one-to-one correspondence between irreducible unitary representations of the fundamental group of $X$ (or, alternatively, irreducible unitary local systems on $X$) and stable vector bundles of degree zero on $X$. 

The Abel--Jacobi theorem has a natural analogue on a metric graph $\Ga$ (see~\cite{MikhalkinZharkov, BakerFaber}), as recalled in Section~\ref{section_AbelJacobi} below. In Section~\ref{section_NarasimhanSeshadri} we prove an analogue of the Narasimhan--Seshadri correspondence in our framework. 

A section of the sheaf $S_n\ltimes \calH_{\Ga}^n$ can be viewed as a $\GL_n(\TT)$-valued function on $\Ga$, and an element of $H^1(S_n\ltimes \calH_{\Ga}^n)$ defines a rank-$n$ vector bundle on $\Ga$. We interpret sections of the natural subsheaf $S_n\ltimes \RR^n$ of locally constant functions as functions valued in a tropical \emph{unitary group}, by showing that the corresponding vector bundles are semistable. 
Indeed, let $\lambda$ be a local system on $\Ga$ with fiber $S_n\ltimes \R^n$. This datum produces a free degree $n$ cover $f:\tGa\to \Ga$ as well as a local system $\widetilde{\lambda}$ with fiber $\R$ (both unique up to isomorphism) such that $f_\ast \widetilde{\lambda}=\lambda$. At the same time, $\lambda$ also naturally defines a vector bundle $E(\lambda)$ with constant transition maps.

\begin{maintheorem}[Narasimhan--Seshadri correspondence, Theorem~\ref{thm_NarasimhanSeshadri}]\label{mainthm_NarasimhanSeshadri}
Let $\Gamma$ be a compact and connected metric graph. 
\begin{enumerate}[(i)]
    \item  A vector bundle $E$ of rank $n$ on $\Gamma$ is associated to an $S_n\ltimes\R^n$-local system $\lambda$ if and only if it is semistable and of degree zero. The vector bundle $E(\lambda)$ is stable if and only if the corresponding $S_n$-representation of $\pi_1(\Ga)$ is indecomposable.
    \item Two $S_n\ltimes\R^n$-local systems $\lambda_i$ (for $i=1,2$) give rise to the same vector bundle if and only if they define the same cover $f\colon\widetilde \Gamma\rightarrow\Gamma$ and the induced classes of the $\widetilde{\lambda}_i$ in $\Jac(\widetilde{\Gamma})$ are equal.
\end{enumerate}
\end{maintheorem}

We point out that all representations of $\pi_1(\Gamma)$ into $S_n\ltimes \R^n$ are finite direct sums of irreducible representations. So, contrary to the classical situation, we do not need to assume semisimplicity on the representation side for our tropical result to hold.

\subsection*{The process of tropicalization} Our approach to vector bundles on a metric graph is completely independent on any kind of base field. In general, however, the notion of a vector bundle on an algebraic curve strongly depends on the chosen base field. So one should not expect there to be a natural tropicalization of an arbitrary algebraic to a tropical vector bundle, at least as defined here. We can see this via a simple dimension count. As we will see in Section~\ref{sec:GLtrop}, the tropical linear group $\GL_n(\TT)$ has real dimension $n$, not $n^2$, and for this reason the moduli space of rank-$n$ vector bundles on a metric graph $\Ga$ of genus $g$ has dimension
$$
\dim_{\RR} \Bun_n(\Ga)=n(g-1)+1\ .
$$
On the other hand, for a projective algebraic curve $X$ the non-stacky dimension of the moduli space of rank $n$ vector bundles is $n^2(g-1)+1$. Hence for $g\geq 2$ we cannot expect the tropicalization map to be defined for all vector bundles. Instead, we consider vector bundles that are linearized with respect to the $n$-dimensional subgroup $S_n\ltimes\G_m^n\subset \GL_n$ of generalized permutation matrices.

Let $X$ be a Mumford curve over an algebraically closed non-Archimedean field $K$ and consider a non-Archimedean skeleton $\Gamma_X$ of the Berkovich analytic space $X^{an}$ in the sense of~\cite{Berkovich_book}. Denote by $\Bun_{S_n\ltimes\G_m^n}(X)$ the moduli stack of $S_n\ltimes\G_m^n$-linearized vector bundles of rank $n$ on $X$ and $\Bun_{S_n\ltimes\G_m^n}^{\free}(X)$ for the open substack of \emph{freely $S_n\ltimes\G_m^n$-linearized} vector bundles of rank $n$ on $X$ (see Section~\ref{section_linearizedvectorbundles} below). Write $\Bun_{S_n\ltimes\G_m^n}^{\free}(X)^{\an}$ for the Berkovich analytification of $\Bun_{S_n\ltimes\G_m^n}^{\free}(X)$ and $\Bun_n(\Gamma)$ for the moduli space of vector bundles on $\Gamma_X$, as introduced in Section~\ref{section_tropicalmoduli}. We construct a natural tropicalization map 
\begin{equation*}
    \trop\colon\Bun_{S_n\ltimes\G_m^n}^{\free}(X)^{\an}\longrightarrow\Bun_n(\Gamma_X)
\end{equation*}
that, on $K$-valued points, associates to a freely $S_n\ltimes\G_m^n$-linearized vector bundle on $X$ a vector bundle on $\Gamma_X$ (essentially by taking valuations of the transition maps).

From a different perspective, the rigidification of the moduli stack $\Bun_{S_n\ltimes\G_m^n}^{\free}(X)$ naturally admits the structure of a disjoint union of a finite quotient of torsors over abelian varieties. So, using Raynaud's non-Archimedean uniformization, we find that there is a natural strong deformation retraction map 
\begin{equation*}\rho\colon \Bun_{S_n\ltimes \G_m^n}^{\free}(X)^{\an}\longrightarrow  \Sigma\big(\Bun_{S_n\ltimes \G_m^n}^{\free}(X)\big)
\end{equation*}
onto a closed subset of $\Bun_{S_n\ltimes \G_m^n}^{\free}(X)^{\an}$, its \emph{non-Archimedean skeleton} (see~\cite[Section 6.5]{Berkovich_book} for details). These two constructions are naturally compatible.

\begin{maintheorem}[Theorem \ref{thm_Buntrop=skel}]\label{mainthm_skel=trop}
Let $X$ be a Mumford curve over $K$. There is a natural isomorphism
\begin{equation*}
    J\colon \Bun_n(\Gamma_X)\xlongrightarrow{\sim} \Sigma\big(\Bun_{S_n\ltimes \G_m^n}^{\free}(X)\big)\end{equation*}
that makes the diagram 
\begin{equation*}
    \begin{tikzcd}
        & \Bun_{S_n\ltimes \G_m^n}^{\free}(X)^{\an} \arrow[ld,"\trop"'] \arrow[rd,"\rho"] &\\ 
        \Bun_n(\Gamma_X) \arrow[rr,"J"',"\sim"] & & \Sigma\big(\Bun_{S_n\ltimes \G_m^n}^{\free}(X)\big)
    \end{tikzcd}
\end{equation*}
commute. 
\end{maintheorem}

In Section~\ref{section_tropicalization} we also show that the process of tropicalization is naturally compatible with forming direct sums, tensor products, duals, and taking determinants, and that a higher rank version of Baker's specialization~\cite{Baker_specialization} inequality holds.

\subsection*{The case of a Tate curve} Let $K$ be an algebraically closed non-Archimedean field of characteristic zero and $X$ a Tate curve over $K$, a smooth projective curve of genus one, whose analytification is given by $X^{\an}=\G_m^{\an}/q^\Z$ with $\val(q)>0$. Then the minimal non-Archimedean skeleton $\Gamma_X$ is isomorphic to a circle of length $\val(q)$. In this case, our tropicalization map can be extended to the whole good moduli space $M_{n,d}(X)$ (as constructed in~\cite{Tu_semistablebundles} using results of~\cite{Atiyah_ellipticvectorbundles}), 
whose $K$-points are in natural one-to-one correspondence with equivalence classes of semistable vector bundles of rank $n$ and degree $d$ on $X$. This moduli space is isomorphic to a suitable symmetric power of $X$, as shown in~\cite{Tu_semistablebundles}, and thus admits a polystable model and a natural non-Archimedean skeleton $\rho\colon M_{n,d}(X)^{\an}\rightarrow \Sigma\big(M_{n,d}(X)\big)$ in the sense of Berkovich~\cite{Berkovich_analytic} (see~\cite{BrandtUlirsch} but also~\cite{Shen_Lefschetz} for details). 

The tropical counterpart of $M_{n,d}(X)$ is the moduli space $M_{n,d}(\Gamma_X)$ of isomorphism classes of semistable vector bundles on $\Gamma_X$, which is explicitly described in Example~\ref{ex:g=1ss}. In Theorem~\ref{thm_Oda} below we use ideas of~\cite{Oda_elliptic},~\cite{PolishchukZaslow}, and~\cite[Theorem 2.18]{BBDG} to prove a characterization of indecomposable vector bundles on a Tate curve that is equivalent to Atiyah's classification. This result allows us to expand on the above construction in Section~\ref{section_Tatecurvetropicalization} in order to construct a natural tropicalization map
\begin{equation*}
    \trop\colon M_{n,d}(X)^{\an}\longrightarrow M_{n,d}(\Gamma_X) 
\end{equation*}
that lands in a component $M_{n,d}^\oplus(\Gamma_X)$ of $M_{n,d}(\Gamma_X)$, which may be identified with a suitable symmetric power of $\Gamma_X$.

\begin{maintheorem}[Theorem~\ref{thm_Tatecurve}]\label{mainthm_Tatecurve}
Let $X$ be a Tate curve. There is a natural isomorphism $J\colon M_{n,d}^\oplus(\Gamma_X)\xrightarrow{\sim}\Sigma\big(M_{n,d}(X)\big)$ that makes the diagram 
\begin{equation*}
    \begin{tikzcd}
        & M_{n,d}(X)^{\an} \arrow[ld,"\trop"'] \arrow[rd,"\rho"] &\\ 
        M_{n,d}^\oplus(\Gamma_X) \arrow[rr,"J"',"\sim"] & & \Sigma\big(M_{n,d}(X)\big)
    \end{tikzcd}
\end{equation*}
commute.
\end{maintheorem}
Theorem~\ref{mainthm_Tatecurve}, in particular, implies that the tropicalization map is continuous, surjective onto $M_{n,d}^\oplus(\Gamma_X)$, and proper, since the retraction $\rho$ has all these properties. In Section~\ref{section_ellipticspecialization}, we prove a refinement of the higher rank specialization inequality. In Sections~\ref{section_BrillNoether} and~\ref{section_generalizedTheta}, this allows us to describe the tropicalization of natural Brill--Noether loci and generalized $\Theta$-divisors in the moduli space $M_{n,0}(X)$ that were introduced classically in~\cite{Tu_semistablebundles}.

In~\cite{KontsevichSoibelman} Kontsevich and Soibelman propose a non-Archimedean incarnation of the SYZ-fibration coming from mirror symmetry. The basic idea is that for a Calabi-Yau variety $X$ the retraction of $X^{\an}$ to its \emph{essential skeleton} in the sense of~\cite{KontsevichSoibelman, MustataNicaise} should function as a natural fibration in affinoid tori. The essential skeleton of the Tate curve is its minimal skeleton and in~\cite[Proposition 6.1.11]{BrownMazzon} the authors show that the formation of essential skeletons is compatible with taking symmetric powers. So Theorem~\ref{mainthm_Tatecurve} also provides us with another explicit example of a non-Archimedean SYZ-fibration for the Calabi-Yau variety $M_{n,d}(X)$. 

By~\cite[Proposition 4.24]{FrancoGarciaPradaNewstead}, the moduli space of semistable Higgs bundles of degree $d$ on an elliptic curve $X$ is naturally isomorphic to a product of $M_{n,d}(X)$ with the symmetric power $\Sym^n(\A^1)$. So Theorem~\ref{mainthm_Tatecurve} also provides us with a non-Archimedean SYZ-fibration for the moduli of semistable Higgs bundles on a Tate curve $X$. We point out that, since $X$ is its own mirror dual and $\GL_n$ is its own Langlands dual, the moduli space of semistable Higgs bundles of rank $n$ will be its own mirror as well.

\subsection*{Further discussion, complements, and remarks}

Theorems~\ref{mainthm_skel=trop} and~\ref{mainthm_Tatecurve} are two new additions to a long line of results identifying tropical moduli spaces with non-Archimedean skeletons of the Berkovich analytic spaces associated to their algebraic counterparts. We refer the reader to~\cite{BakerRabinoff} for the case of Jacobians and to~\cite{ACP} for the moduli space of curves, as the first two examples. These results establish a tight dictionary between tropical and algebraic geometry and lie of the heart of many  applications. They also justify that our synthetic definition of tropical objects (in our case vector bundles) is not just hot air, but is actually related to the corresponding classical notion in a non-trivial way.

In this article we focus on the $\GL_n$-case. We also expect analogues of Theorems~\ref{mainthm_skel=trop} and~\ref{mainthm_Tatecurve} to hold for other reductive groups $G$. The technical input for such a result will most likely come from the work of Laszlo~\cite{Laszlo_elliptic} and Fr{\u{a}}{\c{t}}il{\u{a}}~\cite{Fratila_revisitingelliptic}, where they describe the moduli space of semistable principal $G$-bundles as a moduli space of $T$-linearized principal bundles for a big torus $T$ in $G$ (up to a quotient by the Weyl group). A careful analysis of this situation from a tropical point of view will appear in a follow-up article.

A step towards a more comprehensive theory in higher genus appears to be closely related to the construction of polystable models of the moduli space of vector bundles over a semistable one-parameter degeneration of a smooth projective curve. We refer the reader to the work of Nagaraj and Seshadri~\cite{NagarajSeshadriI,NagarajSeshadriII, Seshadri_Triestelecturenotes} for the already very intricate story, when the special fiber has only one node, building on the work of Gieseker~\cite{Gieseker} in rank two. An underlying problem, which might have prevented further generalizations so far, seems to lie in the fact that, by~\cite{TeodorescuBundlesOnChains, MartensThaddeus_Grothendieck}, an analogue of the Birkhoff--Grothendieck theorem holds on a rational nodal curve only when the dual graph is a chain or a cycle. As soon as the dual graph has vertices of valence $>2$ an immediate generalization of the Birkhoff--Grothendieck theorem is false (see in particular~\cite[Example 4.5]{MartensThaddeus_Grothendieck} for an illustration of the central difficulty). 

While generalizations of our story to the higher genus situation appears to be challenging, a generalization of Theorem~\ref{mainthm_Tatecurve} to the case of abelian varieties seems very much within reach. The crucial point here is that Miyanishi~\cite{Miyanishi} and Mukai~\cite{Mukai} provide us with a classification of homogeneous vector bundles on abelian varieties in the spirit of Atiyah's work~\cite{Atiyah_ellipticvectorbundles}. As a coarse heuristic to explain our expectation, we note that the fundamental group of an abelian variety is abelian, whereas the fundamental group of a curve of genus $g\geq 2$ is not. Tropical vector bundles on abelian varieties might eventually also provide us with a new perspective on the logarithmic approach to Donaldson--Thomas theory introduced in \cite{MaulikRanganathan}.


We have already mentioned that our approach is the same as the one introduced by Allermann in~\cite{TropChern}. In a similar vein, what we propose also appears to be compatible with the theory of tropical vector bundles on a toric tropical scheme developed in~\cite{MinchevaJunTolliver}. A precise relationship would, however, require the theory in~\cite{MinchevaJunTolliver} to be generalized to a local situation (possibly by introducing an analogue of rigid-analytic methods to tropical scheme theory). Once this is achieved, one could then consider $\Z$-invariant vector bundles on $\G_m$ and rediscover vector bundles on the (scheme-theoretic) tropicalization of the Tate curve. For higher genus curves we expect that a more intricate theory will be necessary, e.g.\ expanding on the matroidal notion of tropical ideals coming from \cite{MaclaganRincon}.

\subsection*{Acknowledgements}
The authors would like to thank Chiara Damiolini, Lorenzo Fantini, Emilio Franco, Johannes Horn, Inder Kaur, Enrica Mazzon, Mirko Mauri, Dhruv Ranganathan, Mattia Talpo, Annette Werner, and Jonathan Wise for several helpful discussions en route to and while writing this article.

A.G.\ and M.U.\ have received funding from the Deutsche Forschungsgemeinschaft  (DFG, German Research Foundation) TRR 326 \emph{Geometry and Arithmetic of Uniformized Structures}, project number 444845124, by the Deutsche Forschungsgemeinschaft (DFG, German Research Foundation) Sachbeihilfe \emph{From Riemann surfaces to tropical curves (and back again)}, project number 456557832, and from the LOEWE grant \emph{Uniformized Structures in Algebra and Geometry}.


\section{Line bundles on metric graphs}

In this section we recall the story of divisors and line bundles on metric graphs that was originally introduced in~\cite{MikhalkinZharkov}. Section~\ref{section_linebundles} and~\ref{section_AbelJacobi} put a particular emphasis on the nature of line bundles as torsors over the sheaf of harmonic functions. This perspective is inspired by the logarithmic Picard group, as popularized in~\cite{MolchoWise}, and sets the stage for what is to come in the later sections.

\subsection{Metric graphs} A \emph{star-shaped domain} of valence $n$ and radius $r$ is a metric space of the form
\begin{equation*}
    S(n,r)=\big\{z=te^{\frac{2k\pi i}{n}}\in\C\ \big\vert\  0\leq t\leq r \textrm{ and } k=1,\ldots, n-1\big\} \ ,
\end{equation*}
where the distance between two points is the length of the shortest path connecting them. 
A \emph{metric graph} is a metric space $\Gamma$ such that every point $p\in\Gamma$ has an open neighborhood that is isometric to a star-shaped domain $S(n,r)$ for some $n\geq 0$ and some $r>0$.

Let $G$ be a finite graph (possibly with loops and multiple edges). We write $V(G)$ and $E(G)$ for the sets of vertices and edges of $G$, respectively. An \emph{orientation} on $G$ determines \emph{source} and \emph{target} maps $s,t:E(G)\to V(G)$. Given a length function $\ell\colon E(G)\rightarrow \R_{>0}$, we may construct a compact metric graph $\Gamma(G,\ell)$ by gluing line segments of length $\ell(e)$ for every edge according to the incidences determined by $G$. Given a compact metric graph $\Gamma$, a pair $(G,\ell)$ such that $\Gamma=\Gamma(G,\ell)$ is called a \emph{model} of $\Gamma$; we say that the model is \emph{simple} if $G$ has no loops or multiedges. Given a simple model $(G,\ell)$ of $\Ga$, we define the \emph{star cover} $\calU(G)=\{U_v\}_{v\in V(G)}$ of $\Ga$, where $U_v$ is the open set consisting of $v$ together with the open intervals joining $v$ to all neighboring vertices. Each $U_v$ is contractible, the pairwise intersections $U_v\cap U_{v'}$ are open intervals or empty, and all triple intersections are empty, making the star cover convenient for cohomological calculations.

The \emph{Euler characteristic} of a compact metric graph $\Gamma$ is given by  $\chi(\Gamma)=h^0(\Gamma)-h^1(\Gamma)$; this quantity is additive in connected components. Given a model $G$ of $\Gamma$, the Euler characteristic is given by $\chi(\Gamma)=\#V(G)-\#E(G)$. If $\Gamma$ is connected, we call the number
\[
g=1-\chi(\Gamma)=h^1(\Ga)=\#E(G)-\#V(G)+1
\]
the \emph{genus} of $\Ga$.

A \emph{free cover} of a metric graph $\Gamma$ is an isometric covering space $f\colon\widetilde{\Gamma}\rightarrow\Gamma$. For a free cover $f\colon\widetilde{\Gamma}\rightarrow\Gamma$ of degree $d$, we naturally have
\begin{equation*}
    \chi(\widetilde{\Gamma})= d\cdot \chi(\Gamma) \ .
\end{equation*}
This is a special case of the tropical Riemann-Hurwitz formula.

\begin{remark}
In many approaches to tropical covers, e.g.\ in~\cite{ABBRI, ABBRII} or~\cite{CavalieriMarkwigRanganathan_tropadmissiblecovers}, the authors work with more general \emph{harmonic or admissible covers} of metric graphs that allow for dilation along vertices and edges. We explicitly refrain from including this extra level of intricacy, since, for different reasons, neither Theorem~\ref{mainthm_NarasimhanSeshadri} nor Theorem~\ref{mainthm_skel=trop} appear to generalize to this situation and it is not needed for Theorem~\ref{mainthm_Tatecurve}. For us the  natural next step is the case of logarithmic curves, expanding on~\cite{MolchoWise}, or, equivalently, of metrized curve complexes in the sense of~\cite{AminiBaker}.
\end{remark}

\subsection{Piecewise linear functions, divisors, and the Picard group} Let $\Ga$ be a metric graph. A \emph{piecewise linear function} on an open subset $U\subseteq\Gamma$ is a continuous function $f\colon U\rightarrow \R$ whose restriction to a line segment is piecewise linear with integer slopes. We denote by $\PL_{\Gamma}(U)$ the abelian group of piecewise linear functions on $U$, and observe that the association $U\mapsto\PL_{\Gamma}(U)$ defines a sheaf $\PL_{\Gamma}$ of abelian groups on $\Gamma$. A global section of $\PL_{\Gamma}$ is called a \emph{rational function} on $\Ga$, and we write $\Rat(\Ga)=\PL_{\Ga}(\Ga)$. 

A \emph{divisor} on an open subset $U\subset \Ga$ is a formal sum $D=\sum_{p\in U}D(p)\cdot p$ over the points $p$ of $U$ such that the set $\{p\in U\vert D(p)\neq 0\}$ is discrete. We denote by $\Div(U)$ the abelian group of divisors on $U$, this defines a sheaf $\Div$ of abelian groups on $\Ga$. When $\Gamma$ is compact, we define the \emph{degree homomorphism} 
\begin{equation}
    \begin{split}
        \deg\colon \Div(\Gamma)&\longrightarrow\Z\\
        D&\longmapsto \sum_{p\in\Gamma}D(p) \ .
    \end{split}
\end{equation}
There is a natural sheaf homomorphism $\div\colon \PL_\Gamma\rightarrow\Div$ given by
\begin{equation}
    \begin{split}
         \PL_{\Ga}(U)&\longrightarrow\Div(U)\\
        f&\longmapsto \sum_{p\in\Gamma}\ord_p(f)\cdot p \ ,
    \end{split}
\label{eq:div}
\end{equation}
for $U\subseteq \Gamma$ open, where $\ord_p(f)$ is the sum of the outgoing slopes of $f$ at $p$ (note that $\Delta(f)=-\div(f)$, the sum of the incoming slopes of $f$, is the \emph{Laplacian} of $f$). The image of $\Rat(\Ga)$ is the subgroup of \emph{principal divisors} $\PDiv(\Ga)\subseteq \Div(\Ga)$. We say that two divisors $D,D'\in\Div(\Gamma)$ on $\Ga$ are \emph{linearly equivalent}, written as $D\sim D'$, if $D-D'$ is a principal divisor, i.e. if $D-D'\in\div(f)$ for a rational function $f\in\Rat(\Gamma)$. The \emph{Picard group} of $\Gamma$ is the quotient
\begin{equation*}
    \Pic(\Gamma)=\Div(\Gamma)/\PDiv(\Gamma) \ . 
\end{equation*} 
If $\Ga$ is compact, then $\deg(\div(f))=0$ for all $f\in\Rat(\Gamma)$ by continuity, and the degree map descends to $\Pic(\Gamma)$. We write $\Pic_d(\Gamma)$ for the set of linear equivalence classes of divisors of degree $d$ on $\Gamma$; it is naturally a torsor over the group $\Pic_0(\Gamma)$.  

Given a free cover $f:\tGa\to \Ga$ of degree $d$ and divisors $\widetilde{D}\in \Div(\tGa)$ and $D\in \Div(\Ga)$, we define the \emph{pushforward} and \emph{pullback} as
\[
f_*(\widetilde{D})=\sum_{\widetilde{p}\in\tGa}\widetilde{D}(\widetilde{p})\cdot f(\widetilde{p})\in \Div(\Ga)\quad \textrm{ and } \quad f^*(D)=\sum_{p\in \Ga} D(p)\cdot\sum_{\widetilde{p}\in f^{-1}(p)}\widetilde{p}\in \Div(\tGa) 
\]
respectively. The maps $f_*$ and $f^*$ respect linear equivalence, and hence descend to maps $f_*:\Pic(\tGa)\to \Pic(\Ga)$ and $f^*:\Pic(\Ga)\to \Pic(\tGa)$. For $f^\ast$ this is elementary and for $f_\ast$ we refer to \cite[Section 1.4]{LenUlirsch}.

\subsection{Harmonic functions and line bundles}\label{section_linebundles} The kernel of the divisor map~\eqref{eq:div} is called the \emph{sheaf of harmonic functions} on $\Ga$ and is denoted $\calH_\Gamma$. Any divisor on an open tree is principal, so $\Div=\PL_{\Ga}/\mc H_{\Ga}$ and we have a short exact sequence
\begin{equation}
0\longrightarrow \calH_{\Ga}\longrightarrow \PL_{\Ga}\xlongrightarrow{\div} \Div \longrightarrow 0 \ 
\label{eq:ses1}
\end{equation}
of sheaves of abelian groups on $\Ga$. The sheaves $\calH_\Gamma$ and $\PL_{\Gamma}$ play the role of respectively the sheaves $\calO_X^\ast$ and $\calM_X^\ast$ of invertible holomorphic and meromorphic functions on a Riemann surface $X$. We now extend this analogy by describing the relationship between divisors and line bundles on $\Ga$.

A \emph{line bundle} $\mc L$ on $\Ga$ is an $\mc H_\Gamma$-torsor, that is to say a sheaf of $\mc H_{\Ga}$-sets such that $\Ga$ can be covered by open subsets $U$ for which $\mc L|_U$ and $\mc H_U=\mc H_{\Ga}|_U$ are isomorphic as sheaves of $\mc H_U$-sets. The set of isomorphism classes of line bundles is the sheaf cohomology group $H^1(\Ga,\mc H_{\Ga})$, and the group structure can be described as follows. Let $\mc L$ and $\mc L'$ be two line bundles on $\Ga$. For an open subset $U\subset \Ga$, consider the diagonal action of $\mc H_\Ga(U)$ on $\mc L(U)\times \mc L'(U)$ given by $f(s,s')= (fs,(-f)s')$, and equip the quotient $\big[\mc L(U)\times \mc L'(U)\big]/{\mc H_{\Ga}(U)}$ with the $\mc H_{\Ga}(U)$-action given by $f\overline{(s,s')}=\overline{(fs,s')}$. The assignment $U\mapsto \big[\mc L(U)\times \mc L'(U)\big]/{\mc H_{\Ga}(U)}$ defines an $\mc H_{\Ga}$-torsor called the \emph{tensor product} $\mc L\otimes \mc L'$. Similarly, given a line bundle $\mc L$, we define $\mc L^{-1}$ as the $\mc H_{\Ga}$-torsor whose underlying sheaf of sets is $\mc L$, but equipped with the opposite action $(f,s)\mapsto(-f)s$. It is clear that the map $f\mapsto \overline{(fs,s)}$ defines an isomorphism $\mc H_{\Ga}$-torsors $\mc H_{\Ga}\simeq \mc L\otimes \mc L^{-1}$. The identity element and group operation on $H^1(\Ga,\mc H_{\Ga})$ correspond to the isomorphism class of $\mc H_{\Ga}$ and the tensor product, respectively.

We can also describe a line bundle $\mc L$ on $\Ga$ in terms of a \v{C}ech cocycle as follows. Choose an oriented simple model $(G,\ell)$ of $\Ga$, and let $\calU(G)=\{U_v\}_{v\in v(G)}$ be the associated star cover. Since each $U_v$ is contractible, we can find isomorphisms $g_v:\mc L|_{U_v}\to \mc H_{U_v}$ of sheaves of $\mc H_{U_v}$-sets. Each edge $e\in E(G)$ with start and end vertices $s(e)$ and $t(e)$ corresponds to a nonempty pairwise intersection $U_{s(e)}\cap U_{t(e)}$, which we also denote by $e$ by abuse of notation, and all other pairwise intersections are empty. Hence the line bundle $\mc L$ determines a cocycle $\{g_e\}_{e\in E(\Ga)}$, where $g^e$ is the harmonic function on $e$ (that is to say, a linear function with integer slope) defined by the composition $g^e=\big[g_{t(e)}\circ g_{s(e)}^{-1}\big]\big\vert_{e}:\mc H_e\to \mc H_e$. Two cocycles $\{g_e\}$ and $\{g'^e\}$ determine the same line bundle if and only if $g^e-g'^e=f_{t(e)}-f_{s(e)}$ for some harmonic functions $f_v\in \mc H_{\Ga}(U_v)$, and the cocycle condition is trivially satisfied because all triple intersections are empty. We note that changing the orientation of an edge $e$ replaces $g^e$ with $-g^e$.

Given a line bundle $\mc L$ defined by a \v{C}ech cocycle $\{g^e\}_{e\in E(\Ga)}$ with respect to an oriented simple model $(G,\ell)$, the \emph{degree} of $\mc L$ is defined as
\[
\deg \mc L=\sum_{e\in E(G)}\dot{g}^e,
\]
where $\dot{g}^e$ denotes the slope of $g^e$ in the direction of the oriented edge $e$. Is is easy to verify that $\deg \mc L$ does not depend on the choice of cocycle $\{g^e\}$ or model $G$, and that $\deg \mc L_1\otimes \mc L_2=\deg \mc L_1+\deg \mc L_2$ and $\deg \mc L^{-1}=-\deg \mc L$.

\begin{proposition}\label{prop_linebundle=divisorclass}
Let $\Gamma$ be a metric graph. There is a natural isomorphism
\begin{equation*}
    \Pic(\Gamma)\xlongrightarrow{\sim}H^1(\Gamma,\calH_\Gamma) \ .
\end{equation*}
\end{proposition}

\begin{remark}
In~\cite{MikhalkinZharkov} the term \emph{line bundle} is used for a locally trivial fibration with fiber the tropical affine line over $\Gamma$ in the category integral affine manifolds. In this language, Proposition~\ref{prop_linebundle=divisorclass} has appeared as~\cite[Proposition 4.6]{MikhalkinZharkov}.
\end{remark}

\begin{proof}[Proof of Proposition~\ref{prop_linebundle=divisorclass}]
Given a divisor $D\in\Div(\Gamma)$, we define the associated line bundle $\calH_\Gamma(D)$ by
\begin{equation*}
U\longmapsto \calH_\Gamma(D)(U)=\big\{s\in\PL_\Ga(U)\ \big\vert\ \div(s)=D\vert_U\big\} 
\end{equation*}
for $U\subseteq \Gamma$ open. For $s\in\calH_{\Ga}(D)(U)$ and $f\in\calH_\Gamma(U)$ we have 
\begin{equation*}
    \div(s+f)=\div(s)+\div(f)=\div(s) = D\vert_U \ ,
\end{equation*}
since $f$ is harmonic, so $\mc H_{\Ga}(D)$ is a sheaf of $\mc H_{\Ga}$-sets. 

Let $G$ be an oriented simple model of $\Ga$, and let $U_v\in \calU(G)$ be an element of the star cover of $\Ga$. Since $U_v$ is a tree, we can find a function $f_v\in \PL_{\Ga}(U_v)$ such that $\div f_v=D\vert_{U_v}$. Therefore we have an isomorphism $\calH_\Gamma(D)\vert_{U_v}\xrightarrow{\sim}\calH_{U_v}$ of $\calH_{U_v}$-torsors given by 
\begin{equation*}
    s\longmapsto s-f_v\vert_V
\end{equation*}
for $s\in\calH_\Gamma(D)(V)$ and $V\subseteq U_v$ open. Therefore $\calH_\Gamma(D)$ is an $\calH_\Gamma$-torsor on $\Gamma$. It is easy to verify that $\mc H_{\Ga}(D_1)\otimes \mc H_{\Ga}(D_1)\simeq \mc H_{\Ga}(D_1+D_2)$ and $\mc H_{\Ga}(D)^{-1}\simeq \mc H_{\Ga}(-D)$, so the map $\Div(\Gamma)\rightarrow H^1(\Gamma,\calH_\Gamma)$ given by $D\mapsto \calH_\Gamma(D)$ is a homomorphism, and furthermore $\deg \mc H_{\Ga}(D)=\deg D$.

Let $D_1,D_2\in\Div(\Gamma)$ be two linearly equivalent divisors on $\Gamma$, so that there is a rational function $f\in\Rat(\Gamma)$ such that $D_1=D_2+\div(f)$. The bijection 
\begin{equation*}\begin{split}
  \calH_\Gamma(D_1)(U)&\xlongrightarrow{\sim} \calH_\Gamma(D_2)(U)\\
  s&\longmapsto s-f\vert_U
\end{split}\end{equation*}
for $U\subseteq\Gamma$ open is $\calH_\Gamma(U)$-equivariant and compatible with restriction to open subsets $V\subseteq U$; thus it defines an isomorphism $\calH_\Gamma(D_1)\xrightarrow{\sim}\calH_\Gamma(D_2)$ of $\calH_\Gamma$-torsors. Therefore the map $\Div(\Gamma)\rightarrow H^1(\Gamma,\calH_\Gamma)$ descends to a homomorphism $\Pic(\Gamma)\rightarrow H^1(\Gamma,\calH_\Gamma)$

In order to see that this map is an isomorphism, we consider the long exact sequence associated to the short exact sequence~\eqref{eq:ses1}
\[
0 \longrightarrow \calH_{\Ga}(\Ga) \longrightarrow \Rat(\Ga)\xlongrightarrow{\div} \Div(\Gamma) \xlongrightarrow{\delta} H^1(\Gamma,\calH_{\Ga}) \ .
\]
By definition, $\Pic(\Ga)$ is the cokernel of $\div$, and furthermore the boundary homomorphism $\delta$ is precisely the map $D\mapsto \calH(D)$ described above. By~\cite[Proposition 4.7]{MikhalkinZharkov}, the morphism $\delta$ is surjective, which implies that there is an induced isomorphism $\Pic(\Gamma)\to H^1(\Gamma, \calH_{\Ga})$. \end{proof}

\subsection{The Abel-Jacobi theorem}\label{section_AbelJacobi} 
Let $\Gamma$ be a metric graph. We define the sheaf $\Omega_{\Gamma}$ of \emph{harmonic $1$-forms} on $\Gamma$ by the short exact sequence
\begin{equation}\label{eq_ses1forms}
    0\longrightarrow \R\longrightarrow \calH_\Gamma\longrightarrow \Omega_\Gamma\longrightarrow 0 
\end{equation}
of sheaves of abelian groups. Choose an oriented simple model $(G,\ell)$ of $\Gamma$, and let $s,t:E(G)\to V(G)$ be the associated source and target maps. A harmonic $1$-form $\alpha\in \Omega_{\Gamma}(\Gamma)$ may be written as a formal sum
\begin{equation*}
    \alpha=\sum_{e\in E(G)}a_e de
\end{equation*}
with $a_e\in\Z$, such that at every vertex $v\in V(G)$ we have
$$
\sum_{e \textrm{ with } s(e)=v}a_e=\sum_{e \textrm{ with } t(e)=v}a_e.
$$
Hence, for a compact metric graph $\Ga$, we have a natural identification between $H^0(\Gamma,\Omega_\Gamma)$ and the simplicial homology group $H_1(G,\ZZ)$, which, in turn, is naturally isomorphic to $H_1(\Gamma,\ZZ)$. 

\begin{lemma}\label{lemma_H1diffforms} Let $\Gamma$ be a connected metric graph. Then there is a natural degree isomorphism \[
\deg:H^1(\Gamma,\Omega_{\Gamma}) \xlongrightarrow{\sim}\ZZ.
\]
\end{lemma}

\begin{proof} Choose an oriented simple model $(G,\ell)$ of $\Gamma$ with $n=\#V(G)$ vertices and $m=\#E(G)$ edges. As for line bundles, we compute $H^1(\Ga,\Omega_{\Ga})$ as the first \v Cech cohomology group with respect to the star cover $\calU(G)=\{U_v\}_{v\in v(G)}$. Since all triple intersections are empty, we have an exact sequence
\begin{equation*}
    0\longrightarrow \Omega_\Gamma(\Gamma)\longrightarrow \prod_{v\in V(G)}\Omega_\Gamma(U_v)\longrightarrow \prod_{e\in E(G)} \Omega_\Gamma(e) \longrightarrow \check{H}^1(\Gamma,\Omega_\Gamma) \longrightarrow 0 \ ,
\end{equation*}
where all terms are free abelian groups of finite rank. Define the \emph{degree} of a $1$-cochain $\gamma=\{c^e\}_{e\in E(G)}$ by
\[
\deg \gamma=\sum_{e\in E(G)} c^e.
\]
The degree of a $1$-coboundary is zero, and we claim that the induced map $\deg:\check{H}^1(\Gamma,\Omega_\Gamma)\to \Z$ is an isomorphism. It is clearly surjective, and to complete the proof we compute the ranks in the exact sequence. Since $\Omega_\Gamma(\Gamma)\simeq H_1(\Gamma, \Z)$ we have
\begin{equation*}
    \rk \Omega_\Gamma(\Gamma) = g(\Ga) = m-n+1 \ .
\end{equation*}
Let $v\in V(G)$, then an element $\Omega_{\Gamma}(U_v)$ is determined by $\val (v)$ integers subject to a single relation, so $\rk \Omega_{\Gamma}(U_v)=\val(v)-1$ and therefore
\begin{equation*}
        \rk \prod_{v\in V(G)}\Omega_\Gamma(U_v) =
        \sum_{v\in V(G)} (\val(v)-1)=
        2m-n \ .
\end{equation*}
Finally, a harmonic $1$-form on an edge is simply an integer, so $\Omega_{\Gamma}(e)\simeq \ZZ$ and hence
\begin{equation*}
        \rk \prod_{e\in E(G)} \Omega_\Gamma(e)=m \ .
\end{equation*}
Putting all this together, we see that
\begin{equation*}
    \rk \check{H}^1(\Gamma,\Omega_\Gamma)= (m-n+1)-(2m-n)+m=1 \ ,
\end{equation*}
which completes the proof.
\end{proof}

\begin{remark} This result appears to be well-known. In the framework of tropical homology (as introduced by~\cite{IKMZ}), Poincar\'e duality establishes a natural isomorphism $H^1(\Ga,\Omega_{\Ga})\simeq H_{0,0}(\Ga,\ZZ)$, which is isomorphic to $\ZZ$ via the degree map (see~\cite{Lefschetz}). Alternatively one could also deduce this from~\cite[Section 3.4]{MolchoWise}.

\end{remark}

Let us continue to assume that $\Gamma$ is compact and connected. The short exact sequence \eqref{eq_ses1forms} provides us with a long exact sequence
\begin{equation}
    0\longrightarrow \R\longrightarrow \calH_\Gamma(\Gamma)\longrightarrow \Omega_\Gamma(\Gamma)\longrightarrow H^1(\Gamma,\R)\longrightarrow \Pic(\Gamma)\xlongrightarrow{\deg} \Z\rightarrow 0 \ ,
\label{eq:les}
\end{equation}
where the second arrow is an isomorphism and we identify the last arrow with the degree homomorphism using Proposition~\ref{prop_linebundle=divisorclass} and Lemma~\ref{lemma_H1diffforms}. Hence the middle arrow establishes an isomorphism between $\Pic_0(\Ga)$ and the real $g$-dimensional torus $H^1(\Ga,\RR)/\Omega_{\Ga}(\Ga)$ known as the \emph{Jacobian} $\Jac(\Ga)$. We note that according to some authors, $H^1(\Ga,\RR)/\Omega_{\Ga}(\Ga)$ is the \emph{Albanese variety} $\Alb(\Ga)$, while $\Jac(\Ga)$ is the dual torus $\Omega_{\Ga}(\Ga)^*/H_1(\Ga,\ZZ)$. 

This isomorphism also admits a description in terms of an Abel--Jacobi map $\Gamma\rightarrow\Jac(\Gamma)$ inducing a homomorphism $\Div_0(\Gamma)\rightarrow\Jac(\Gamma)$ and subsequently an isomorphism $\Pic_0(\Gamma)=\Div_0(\Gamma)/\PDiv(\Gamma)\xrightarrow{\sim}\Jac(\Gamma)$; for more details we refer the interested reader to~\cite{MikhalkinZharkov} and~\cite{BakerFaber}.


\section{Principal bundles and vector bundles on metric graphs}

\subsection{Root data of reductive groups}
It is well known that the category of split algebraic tori is naturally equivalent to the category of finitely generated free abelian groups in two (dual) ways: to go from a split algebraic torus to a finitely generated free abelian group we either take the character lattice or the cocharacter lattice. 

Split connected reductive algebraic groups over a fixed field are classified in terms of root data (see~\cite[Ch.\ 7 and 8]{Springer_linalggroups} for details). 
A \emph{root datum} is a quadruple $\Phi=(M,R,M^\vee,R^\vee)$ consisting of 
\begin{itemize}
    \item a finitely generated free abelian group $M$ and its dual group $M^\vee=\Hom(M,\Z)$ (with duality pairing $\langle.,.\rangle$) and
    \item a subset of \emph{roots} $R$ of $M$ and of \emph{coroots} $R^\vee$ of $M^\vee$ together with a bijection $(.)^\vee\colon R\rightarrow R^\vee$
\end{itemize}
subject to the following two axioms:
\begin{enumerate}[(i)]
    \item For all $\alpha\in R$ we have $\langle \alpha,\alpha^\vee\rangle=2$;
    \item The reflection homomorphisms $s_\alpha\colon M\rightarrow M$ and $s_\alpha^\vee\colon M^\vee\rightarrow M^\vee$ given by 
    \begin{equation*}
        u\longmapsto u-\langle u,\alpha^\vee\rangle \alpha \qquad \textrm{and} \qquad v\longmapsto v-\langle \alpha,v\rangle \alpha
    \end{equation*}
    fulfill
    \begin{equation*}
        s_\alpha(R)=R \qquad \textrm{and} \qquad s_\alpha^\vee(R^\vee)=R^\vee
    \end{equation*}
    for all $\alpha\in R$.
\end{enumerate}
Unlike the case of tori, this classification is not a categorical equivalence. When $G$ is semisimple the associated root datum is nothing but a root system and this recovers the well-known classification of semisimple Lie algebras. 

The \emph{Weyl group} $W=W_\Phi$ of a root datum $\Phi$ is the (finite) automorphism group of $M$ generated by all reflections $s_\alpha$ for $\alpha\in R$. 

Our main example is the group $\GL_n$. In this case $M=\Z^n$ is the character lattice of its diagonal torus $\G_m^n$, its dual lattice is the cocharacter lattice of $\G_m^n$ and the roots (and coroots) are given by 
\begin{equation*}
    \big\{e_i-e_j\ \big\vert\  1\leq i,j\leq n \quad \textrm{and}\quad  i\neq j\big\}
\end{equation*}
where $e_1,\ldots,e_n$ denote the standard basis vectors of $\Z^n$. The Weyl group is isomorphic to the symmetric group $S_n$ operating via permutation matrices. The character lattice of the group $
\SL_n$ is the sublattice $M_0\subset \ZZ^n$ consisting of vectors whose coordinates sum to zero, while the Weyl group $W=S_n$ is the same. 

\subsection{Principal bundles on metric graphs.} Let $\Ga$ be a metric graph, let $\calH_{\Gamma}$ be the sheaf of harmonic functions on $\Gamma$, and let $\Phi=(M,R,N,R^\vee)$ be a root datum with Weyl group $W$. For an open subset $U\subset \Ga$, the $W$-action on $M$ extends to a $W$-action on $\calH_\Gamma(U)\otimes_\Z M$, where $W$ acts trivially on the first summand. Taking the semidirect product, we obtain a sheaf $W\ltimes \calH_\Gamma\otimes_\Z M$ of (generally speaking) non-abelian groups on $\Ga$.

\begin{definition}
Let $\Gamma$ be a metric graph and let  $\Phi=(M,R,M^\vee,R^\vee)$ be a root datum with Weyl group $W$. A \emph{principal $\Phi$-bundle}  on $\Gamma$ is a $W\ltimes \calH_\Gamma\otimes_\Z M$-torsor on $\Gamma$.
\end{definition}

When $G$ is a connected split reductive group and $\Phi$ is the associated root datum, we also sometimes refer to a principal $\Phi$-bundle as a \emph{principal $G$-bundle} on $\Gamma$. In the case $G=\GL_1=\G_m$, a principal $\G_m$-bundle is nothing but a $\calH_\Gamma$-torsor on $\Gamma$. So, in this case, we recover the theory of line bundles on metric graphs outlined in Section~\ref{section_linebundles} above. 

\begin{remark}
Let $G$ be a connected split reductive group over a field $k$. Our definition of a tropical analogue of principal $G$-bundle on a metric graph $\Gamma$ is completely independent of $k$. So one may think of this as being defined over the ''field $\FF_1$ with one element``. This goes back to a suggestion by Tits~\cite{Tits_F1}. For the  state of the art concerning analogues of reductive groups over $\FF_1$ we refer the reader to~\cite{Lorscheid_blueprintsII}. 
\end{remark}

The set of isomorphism classes of principal $\Phi$-bundles over a metric graph $\Ga$ is given by the non-abelian \v{C}ech cohomology $\check{H}^1(\Gamma,W\ltimes\calH_\Gamma\otimes M)$. This set is convenient to describe in terms of an oriented simple model $(G,\ell)$ of $\Ga$. Indeed, let $E$ be a principal $G$-bundle on $\Ga$, then we can find trivializations $g_v:E|_{U_v}\to W\ltimes \calH_\Gamma(U_v)\otimes_\Z M$, and $E$ is described in terms of the transition functions
\[
g^e=\big[g_{t(e)}\circ g^{-1}_{s(e)}\big]:W\ltimes \calH_\Gamma(e)\otimes_\Z M\longrightarrow W\ltimes\calH_\Gamma(e)\otimes_\Z M.
\]
Choosing a basis $M\simeq \ZZ^n$, we can write the transition functions as $g^e=(w^e,g_1^e,\ldots,g_n^e)$, where $w^e\in W$ and the  $g_i^e\in \mc H_{\Ga}(e)$ are linear functions with integer slope. The cocycle condition is trivially verified since triple intersections are empty.

\subsection{Vector bundles on metric graphs.} \label{sec:GLtrop} From now on we will focus on principal $\GL_n$-bundles on a metric graph $\Gamma$ or, in other words, $S_n\ltimes\calH_\Gamma^n$-torsors on $\Gamma$. To avoid clunky terminology, we will refer to such an object simply as a \emph{vector bundle} of rank $n$ on $\Gamma$.

It is convenient to represent transition functions of vector bundles on $\Ga$ in terms of the tropical semifield $\mathbb T=\R\cup\{\infty\}$ as follows. By~\cite[Lemma 1.4]{TropChern}, the invertible $n\times n$ matrices over $\mathbb T$ are precisely the compositions of invertible diagonal matrices (having finite entries on the diagonal and $\infty$ elsewhere) and permutation matrices. In other words, $\GL_n(\mathbb T)=S_n\ltimes \R^n$ is the group of matrices having a unique finite entry in each row and each column, while $\SL_n(\mathbb T)\subset\GL_n(\mathbb T)$ is the subgroup of matrices whose \emph{tropical determinant}, defined as the sum of the finite entries, is zero. 

Given a vector bundle $E$ of rank $n$ on a metric graph $\Ga$, we can represent its transition functions as $\GL_n(\mathbb T)$-valued harmonic functions as follows. Choose an oriented simple model, and let $g^e=(\sigma^e,g^e_1,\ldots,g^e_n)$ for $e\in E(G)$ be the transition functions of $E$, where $\sigma^e\in S_n$ and $g^e_i\in \mc H_{\Gamma}(e)$. We can view the $g^e$ as $\GL_n(\TT)$-matrices in the following way:
\[
g^e_{ij}=\left\{\begin{array}{cc} g^e_i,& \textrm{if } j=\sigma^e(i),\\
\infty, & \textrm{if } j\neq \sigma^e(i).\end{array}
\right.
\]
Note also that our notion of a tropical vector bundle coincides with the one proposed in~\cite{TropChern}. 

We note that we may construct the \emph{total space} of a vector bundle $E$ on a metric graph $\Ga$ as an integral affine manifold by gluing trivializations $U_v\times \TT^n$ along the tropical linear maps $g^e_{ij}$. A principal $\PGL_n(\TT)$-bundle, where $\PGL_n(\TT)=\GL_n(\TT)/\TT$, then corresponds to a fibration by tropical projective spaces $\TT\PP^{n-1}=\big(\TT^n\backslash\{\infty,\ldots,\infty\}\big)/\TT$. We leave the details of this construction to the avid reader.

\subsection{Basic operations}\label{subseq:basic} We now describe a number of standard constructions for vector bundles in our tropical situation. 
Let $E$ and $F$ be two vector bundles of rank $m$ and $n$ on a metric graph $\Gamma$. Choose an oriented simple model $(G,\ell)$ of $\Ga$, and let $\{g^e_{ij}\}_{e\in E(G)}$ and $\{h^e_{kl}\}_{e\in E(G)}$ for $e\in E(G)$ be the transition functions of $E$ and $F$, respectively.

\begin{enumerate}[(i)]
    \item The \emph{direct sum} $E\oplus F$ is the vector bundle of rank $m+n$ whose transition functions on an edge $e\in E(G)$ are given by the block diagonal matrices
    \begin{equation*}
        \begin{bmatrix}
            g^e_{ij} & \infty\\
            \infty & h^e_{kl}
        \end{bmatrix} \ .
    \end{equation*}
    \item The \emph{tensor product} $E\otimes F$ is the vector bundle of rank $mn$ whose transition functions on an edge $e\in E(G)$ are given by the Kronecker product of the transition functions of $E$ and $F$, in other words by $mn\times mn$ matrices whose $((i,k),(j,l))$-th entry is
    \begin{equation*}
            g^e_{ij} + h^e_{kl} 
        \ .
    \end{equation*}

    \item The \emph{dual vector bundle} $E^\ast$ is the vector bundle of rank $n$ whose transition functions on an edge $e\in E(G)$ are
    \begin{equation*}
   \widetilde{g}^e_{ij}=\begin{cases} -g^e_{ij}, & \textrm{if } g^e_{ij}\neq \infty \ ,\\
    \infty, &  \textrm{if } g^e_{ij}= \infty \ .
    \end{cases}
    \end{equation*} 
    \item The \emph{determinant} $\det(E)$ is the line bundle on $\Gamma$ whose transition functions are given by the tropical determinants of $g^e_{ij}$, in other words the sums of all finite entries.
    \item The \emph{degree} of $E$ is the sum of the slopes of the finite entries of the $g^e_{ij}$ over all edges $e\in E(G)$.
\end{enumerate}

We leave it to the avid reader to verify that these constructions do not depend on the choice of model $G$ and transition functions. We furthermore note that a rank-$n$ vector bundle $E$ admits the structure of an $\SL_n$-bundle if and only if $\det (E)$ is the trivial line bundle. Finally, it is clear that the equations
\[\begin{split}
\deg E\oplus F&=\deg E+\deg F\\ \deg E\otimes F&=\deg E\cdot \deg F\\ \deg E^\ast&=-\deg E\\ \deg \det E&=\deg E
\end{split}\]
hold for all vector bundles $E$ and $F$.

\begin{remark}
There are several different options on how one could define the notion of a \emph{homomorphism} between two vector bundles $E$ and $F$: as a section of $E^\ast\otimes F$, in terms of its total space, or by mimicking morphisms of principal bundles. In fact, there does not even seem to be a generally agreed-upon definition of a tropical linear map (which would immediately provide us with a notion of a homomorphism of trivial bundles). This is why we refrain from opening Pandora's box and prefer not to propose a definite definition at this point. 
\end{remark}

We point out, however, that two vector bundles are isomorphic (as torsors) if and only if their transition functions define the same class in $H^1(\Gamma,S_n\ltimes \calH_\Gamma^n)$. Thus, it is straightforward to define subbundles:

\begin{definition} Let $E$ and $F$ be vector bundles of rank $n$ and $m\leq n$ on $\Gamma$. We say that $F$ is a \emph{subbundle} of $E$, if we may choose transition functions $h_{ij}^e$ of $F$ and $g_{ij}^e$ of $E$ in such a way that 
\begin{equation*}
    \big[g_{ij}^e\big]=\begin{bmatrix}
        h_{ij}^e & \ast\\
        \infty& \ast
    \end{bmatrix}
\end{equation*}
for all $e\in E(G)$.
\end{definition}

Since the matrix $h^e_{ij}$ is invertible, the unique finite entry in each of the first $m$ rows occurs in one of the first $m$ columns. In other words, an invertible block-triangular matrix is block-diagonal, and we have the following result:

\begin{proposition}
\label{prop_subbundlesplit}
 Every subbundle $F$ of a vector bundle $E$ on $\Gamma$ is \emph{split}, i.e. there exists another subbundle $H$ of $E$ such that $E\simeq F\oplus H$. 
\end{proposition}

We also define the pushforward and pullback of vector bundles along a free cover.  Let $f:\tGa\to \Ga$ be a free cover of degree $d$. Choose models $G$ and $\tG$ for $\Ga$ and $\tGa$ such that $f$ corresponds to a graph morphism $\tG\to G$. Choose a labeling $f^{-1}(v)=\{\tv_1,\ldots,\tv_d\}$ of the preimages of each vertex $v\in V(G)$. Similarly, for each edge $e\in E(G)$ we label the preimages $f^{-1}(e)=\{\te_1,\ldots,\te_d\}$, where we assume that $s(\te_k)=\widetilde{s(e)}_k$. The cover is then determined by an $S_d$-valued cocycle $\{\sigma^e\}_{e\in E(G)}$, where for each $e\in E(G)$, the permutation $\sigma^e\in S_d$ is determined by the formula $t(\te_k)=\widetilde{t(e)}_{\sigma^e(k)}$.

\begin{definition} \label{def:pushpull} Let $f:\tGa\to \Ga$ be a free cover of metric graphs of degree $d$ with models $\tG$ and $G$.

\begin{enumerate}[(i)]
    \item Let $E$ be a vector bundle of rank $n$ on $\Ga$, given by transition functions $\{g^e_{ij}\}_{e\in E(G)}$ with respect to the model $G$. We define the \emph{pullback} $f^*E$ as the rank-$n$ vector bundle on $\tGa$ with transition functions
\[
\widetilde{g}^{\te}_{ij}=g^{f(\te)}_{ij}
\]
for $\te\in E(\tG)$.

    \item Let $\widetilde{E}$ be a vector bundle of rank $n$ on $\tGa$ defined by the cocycle $\{\widetilde{g}^{\te}_{ij}\}_{\te\in E(\tG)}$.  The \emph{pushforward} $f_*E$ is the rank-$n\cdot d$ vector bundle on $\Ga$ whose transition function over $e\in E(G)$ is the $n\cdot d$ block matrix whose $(k,l)$-th $n\times n$ block is $\widetilde{g}^{\te_k}_{ij}$ if $l=\sigma^e(k)$ and $\infty$ otherwise.

\end{enumerate}

\end{definition}


\section{Vector bundles and multidivisors}\label{section_multidivisors}


\subsection{Multidivisors}
Line bundles on a metric graph $\Gamma$ may be described in terms of linear equivalence classes of divisors on $\Gamma$. In the higher-rank situation, we have a similar description in terms of so-called \emph{multidivisors}. In this section we establish a dictionary between the two.

\begin{definition}
Let $\Gamma$ be a compact metric graph. 
\begin{enumerate}[(i)]
\item A \emph{multidivisor} $(f\colon\tGa\to \Ga,D)$ on $\Gamma$ of rank $n$ consists of a (possibly disconnected) topological cover $f\colon\widetilde{\Gamma}\rightarrow\Gamma$ of degree $n$ and a divisor $D$ on $\widetilde{\Gamma}$.
\item Two multidivisors $(f\colon\tGa\to \Ga,D)$ and $(f'\colon\tGa'\to \Ga,D')$ are said to be \emph{linearly equivalent} if there is an isometry $\phi\colon \widetilde{\Gamma}\xrightarrow{\sim}\widetilde{\Gamma}'$ of metric graphs such that $f'\circ \phi=f$ and such that the difference $D-\phi^\ast D'$ is a principal divisor on $\widetilde{\Gamma}$.
\end{enumerate}
\end{definition}

One may think of think of the notion of a multidivisor as a tropical incarnation of the Weil's notion of a matrix divisor from~\cite{Weil}. Given a multidivisor $(f,D)$ on a compact metric graph $\Gamma$, the pushforward
\begin{equation*}
    \calH(f,D):=\calH_\Gamma(f,D):=f_\ast\calH_{\widetilde{\Gamma}}(D)
\end{equation*}
is a vector bundle of rank $n$ on $\Gamma$. We now show that these two notions are essentially equivalent.

\begin{proposition}\label{prop_multidivisors}
Let $\Gamma$ be a compact metric graph. The association \begin{equation*}
(f,D)\longmapsto \calH(f,D)
\end{equation*} induces a natural bijection between linear equivalence classes of multidivisors of rank $n$ and isomorphism classes of vector bundles of rank $n$ on $\Gamma$. 

Under this correspondence, the basic operations correspond to the following:
\begin{enumerate}[(i)]
    \item \label{item:multi1} Given two multidivisors $(f_1,D_1)$ and $(f_2,D_2)$, we write $f_1\oplus g_1\colon \widetilde{\Gamma}_1\sqcup\widetilde{\Gamma}_2\rightarrow \Gamma$ for the disjoint union of the covers and $D_1\oplus D_2$ for the induced divisor on $\widetilde{\Gamma}_1\sqcup\widetilde{\Gamma}_2$. Then 
    \begin{equation*}
         \calH\big(f_1\oplus f_2, D_1\oplus D_2\big)\simeq \calH(f_1,D_1)\oplus  \calH(f_2,D_2) \ .
    \end{equation*}
    \item \label{item:multi2} Given two multidivisors $(f_1,D_1)$ and $(f_2,D_2)$, we write $f_1\times_\Gamma f_2\colon \widetilde{\Gamma}_1\times_\Gamma \widetilde{\Gamma}_2\rightarrow\Gamma$ for the fibered product of topological covers and 
    \begin{equation*}
        D_1\boxtimes D_2=\sum_{(p,q)\in \widetilde{\Gamma}_1\times_\Gamma \widetilde{\Gamma}_2} (D_1(p)+D_2(q))(p,q)
    \end{equation*}
    for the induced divisor on $\widetilde{\Gamma}_1\times_\Gamma \widetilde{\Gamma}_2$. Then 
    \begin{equation*}
         \calH\big(f_1\times_\Gamma f_2,D_1\boxtimes D_2\big)\simeq \calH(f_1,D_1)\otimes \calH(f_2,D_2) \ . 
    \end{equation*}
    \item \label{item:multi3} Given a multidivisor $(f,D)$ on $\Gamma$, we have 
    \begin{enumerate}
    \item \begin{equation*}
         \calH(f,-D)\simeq\calH(f,D)^\ast
    \end{equation*}
        as well as
    \item \begin{equation*}
         \det \calH(f,D)\simeq \calH(f_\ast D)
    \end{equation*}
    and
    \item \begin{equation*}
        \deg \calH(f,D)=\deg(D).
    \end{equation*}
    \end{enumerate}
    \end{enumerate}
Furthermore, under this equivalence, indecomposable vector bundles on $\Gamma$ correspond to those multidivisors $(f,D)$ for which $f$ is a connected topological cover. 
\end{proposition}

\begin{proof} Choose a simple model $G$ for $\Ga$. Let $E$ be a rank-$n$ vector bundle on $\Ga$ with transition functions $\{\sigma^e,g^e_1,\ldots,g^n_e\}_{e\in E(G)}$, where $\sigma^e\in S_n$ and $g^i_e\in \mc H(e)$. The $S_n$-cocycle $\{\sigma^e\}_{e\in E(G)}$ determines a free cover $f:\tGa\to \Ga$ of degree $n$, which we describe using the notation from the paragraph preceding Definition~\ref{def:pushpull}. Let $\calL_E$ be the line bundle on $\tGa$ whose transition function on the edge $\te_i$ over $e\in E(G)$ is equal to $g^i_e$. It follows from Definition~\ref{def:pushpull} that $f_*\calL_E\simeq E$. Choosing a divisor $D\in \Div(\tGa)$ such that $\calL_E\simeq \mc H_{\tGa}(D)$, it is elementary to verify that the isomorphism class of the pair $(f,D)$ does not depend on the choices made. The verification of (\ref{item:multi1})-(\ref{item:multi3}) follows immediately by calculating cocycles with respect to the cover $G$ (see the definitions of the basic operation in Section~\ref{subseq:basic}) and is left to the avid reader. 
\end{proof}

\begin{example}[Vector bundles on metric trees]\label{example_BirkhoffGrothendieck} Let $\Gamma$ be a compact and connected metric tree. Two divisors on $\Gamma$ are linearly equivalent if and only if they have the same degree. So, up to isomorphism, there is exactly one line bundle $\calH_{\Ga}(d)$ of degree $d$ on $\Gamma$.

Since the fundamental group of $\Gamma$ is trivial, the only connected covering space of $\Ga$ is $\Gamma$ itself. Hence for each $n$ there is a unique degree $n$ free cover $\widetilde{\Gamma}\rightarrow \Gamma$ consisting of $n$ disjoint copies of $\Gamma$. By Proposition~\ref{prop_multidivisors}, this means that every vector bundle $E$ of rank $n$ on $\Gamma$ splits as a sum of line bundles
\begin{equation*}
    E\simeq \calH_\Gamma(a_1)\oplus \cdots \oplus \calH_\Gamma(a_n)
\end{equation*}
for a unique multiset of $n$ integers $a_1,\ldots, a_n$. This is a tropical analogue of the Birkhoff--Grothendieck theorem which states that any vector bundle on $\PP^1$ is a direct sum of line bundles (see ~\cite{Grothendieck_vectorbundlesonP1}). 
\end{example}


\subsection{The moduli space $\Bun_n(\Gamma)$}\label{section_tropicalmoduli}
 Proposition~\ref{prop_multidivisors} may be used to describe the moduli space $\Bun_n(\Gamma)$ of all vector bundles on a compact and connected metric graph $\Gamma$. One may think of this in analogy with~\cite{NarasimhanSeshadri_moduli}, where a moduli space of (semi)-stable vector bundles is constructed using the Narasimhan-Seshadri correspondence.

\begin{corollary}\label{cor_modulicomponents}
Let $\Gamma$ be a compact and connected metric graph. The set $\Bun_n(\Gamma)$ of isomorphism classes of rank $n$ vector bundles on $\Gamma$ is naturally given by
\begin{equation*}
    \Bun_n(\Gamma)=\bigsqcup_{f\colon \widetilde{\Gamma}\rightarrow\Gamma} \Pic(\widetilde{\Gamma})/\Aut(f) \ ,
\end{equation*}
where the union is taken over isomorphism classes of topological covers $f\colon\widetilde{\Gamma}\rightarrow \Gamma$ of degree $n$, and $\Aut(f)$ is the group of deck transformations of $f$ acting on $\Pic(\Gamma)$ via pullback.
\end{corollary}

\begin{proof}
This is an immediate consequence of Proposition~\ref{prop_multidivisors}.
\end{proof}

For a fixed free cover $f\colon\widetilde{\Gamma}\rightarrow \Gamma$, we write 
\begin{equation*}
    \widetilde{\Gamma}=\widetilde{\Gamma}_1\sqcup \cdots \sqcup \widetilde{\Gamma}_k
\end{equation*} for the decomposition into connected components. We then have
\begin{equation*}
    \Pic(\widetilde{\Gamma})= \Pic(\widetilde{\Gamma}_1)\times \cdots \times \Pic(\widetilde{\Gamma}_k)
\end{equation*}
and every $\Pic_d(\widetilde{\Gamma}_i)$ is naturally a torsor over the real torus $\Pic_0(\widetilde{\Gamma}_i)\simeq\Jac(\widetilde{\Gamma}_i)$. So $\Bun_n(\Gamma)$ is naturally a disjoint union of finite group quotients of torsors over real tori. 

We point out that if $\Gamma$ has genus $g$, then $\widetilde{\Gamma}$ has genus $\widetilde{g}=n(g-1)+1$. Hence we obtain that
\begin{equation*}
    \dim_{\RR} \Bun_n(\Ga)=n(g-1)+1.
\end{equation*}
Classically, the dimension of the moduli space of vector bundles on an algebraic curve of genus $g$ is equal to $d(g-1)+1$, where $d$ is the dimension of the structure group, and in particular the dimension is $n^2(g-1)+1$ when the structure group is $\GL_n$. This discrepancy is explained by noting that $\dim_{\RR}\GL_n(\TT)=n$, not $n^2$, and is further clarified by Theorem~\ref{mainthm_skel=trop}, which states that tropical vector bundles (according to our definition) arise as tropicalizations of geometrically $S_n\ltimes\G_m^n$-linearized vector bundles; so the algebraic structure group actually has dimension $n$, not $n^2$. In the case $g=1$, the two dimensions agree and we have a much stronger tropicalization result in Theorem~\ref{mainthm_Tatecurve}.


\subsection{Vector bundles on tropical elliptic curves}\label{section_Atiyah} Let $\ell>0$ be a real number, and let $\Gamma=\R/\ell\Z$ be a circle of length $\ell$ (for example, $\Gamma$ could be the minimal skeleton of a Tate curve $X^{\an}=\G_m^{\an}/q^\Z$, where $\ell$ is the valuation of $q$). We determine the structure of the set $M_{n,d}^{\ind}(\Ga)$ of isomorphism classes of indecomposable vector bundles of rank $n$ and degree $d$ on $\Ga$. 

Since $\pi_1(\Gamma)=\Z$, for each $n\geq 1$ there is a unique connected free cover $f:\widetilde{\Gamma}\to \Ga$ of degree $n$, where $\tGa=\R/n\ell\Z$ is a circle of length $n\ell$ and the covering map $f\colon \widetilde \Gamma\to \Gamma$ 
is the quotient map $\R/n\ell\Z\to \R/\ell\Z$. For a real number $x\in \RR$, we denote the corresponding points on $\Ga$ and $\tGa$ by $p_x=x+\ell \ZZ$ and $\tp_x=x+n\ell \ZZ$, respectively. The group $\Aut (f)$ is the cyclic group of order $n$ generated by the deck transformation $g(\tp_x)=\tp_{x+\ell}$.

We can identify $\Ga=\Pic^0(\Ga)$ via the map $p_x\mapsto [p_x-p_0]$ (which is in fact an isomorphism of groups), and similarly $\tGa$ is identified with $\Pic^0(\tGa)$. In terms of these identifications, the pullback morphism
 \[
 f^*\colon \Gamma=\Pic^0(\Gamma)\longrightarrow \Pic^0(\widetilde \Gamma)=\widetilde \Gamma
 \]
is given by multiplication by $n$, in other words it is the bijection $p_x-p_0\mapsto \tp_{nx}-\tp_0$. The set $\Pic^d(\tGa)$ is a torsor over $\Pic^0(\tGa)$, and we translate by $d\tp_0$ to obtain the bijection
\begin{equation*}
    \begin{split}
        \Ga &\longrightarrow \Pic^d(\tGa)\\
        p_x& \longmapsto d\tp_0+f^*(p_x)=(d-1)\tp_0+\tp_{nx}.
    \end{split}
\end{equation*}
By Corollary~\ref{cor_modulicomponents} we can identify $M_{n,d}^{\ind}(\Ga)$ with the quotient $\Pic^d(\tGa)/\Aut (f)$. The generator $g$ of $\Aut(f)$ acts on $\Pic^d(\tGa)$ (identified with $\tGa$) by translation by $d\ell$, hence $p_x$ and $p_{x'}$ define the same vector bundle if and only if $x-x'$ is a multiple of $\frac{d\ell}{n}$. The multiples of $\frac{d\ell}{n}$ correspond to the $n'=\frac{n}{\gcd(d,n)}$-torsion points of $\Gamma$ and constitute the kernel of the multiplication by $n'$ map
\begin{equation*}\begin{split}
     \Gamma&\longrightarrow \Gamma\\
     p_x&\longmapsto p_{n'x} \ .
\end{split}\end{equation*}

We now summarize our results for future use and rephrase them in the notation of Proposition~\ref{prop_multidivisors}. Consider the  vector bundle
\[
E^\trop(n,d)= \mc H(f, d\cdot \tp_0)
\]
of degree $d$ and rank $n$ on $\Ga$. Then the formula
\[
\Psi(p_x)=\mc H\big(f,d \tp_0+f^*(p_x-p_0)\big) = E^\trop(n,d)\otimes \mc H_{\Ga}(p_x-p_0) \ . 
\]
defines a surjective map $\Psi:\Gamma \to M_{n,d}^{\ind}(\Ga)$.
Hence the map $\Psi$ fits into the diagram
\begin{equation}
    \label{equation_vector bundles on tropical elliptic curves}
    \begin{tikzcd}
        \Gamma \arrow [r,"\cdot n'"] \arrow[d,"\Psi",swap]  &  \Gamma \arrow[ld, "\simeq"]  \\
        M_{n,d}^{\ind}(\Ga) &
    \end{tikzcd},
 \end{equation}
where $\Ga\xrightarrow{\cdot n'} \Ga$ is the multiplication by $n'$ map, and we have an identification of $M_{n,d}^{\ind}(\Ga)$ with $\Ga$. 

We may view this identification as a tropical analogue of Atiyah's classification of indecomposable vector bundles on an elliptic curve from~\cite{Atiyah_ellipticvectorbundles}.


\section{Baker--Norine rank and a Weil--Riemann--Roch theorem}\label{section_WeilRiemannRoch}

\subsection{Linear systems and Baker--Norine rank}
Let $\Gamma$ be a metric graph. A divisor $D\in\Div(\Gamma)$ is said to be \emph{effective}, written as $D\geq 0$, if $D(p)\geq 0$ for all $p\in\Gamma$. We write $D\geq D'$ for divisors $D,D'\in\Div(\Gamma)$ if $D-D'\geq 0$. The \emph{(complete) linear system} $\vert D\vert$ associated to $D\in\Div(\Gamma)$ is the set
\begin{equation*}
    \vert D\vert=\big\{D'\geq 0\mid D\sim D'\big\} \ .
\end{equation*}

Let $\Gamma$ be a compact metric graph. The \emph{Baker--Norine rank} $r_\Gamma(D)$ of a divisor $D$ on $\Gamma$ was introduced in~\cite{BakerNorine} and is defined as follows. Suppose first that $\Gamma$ is connected. If $\vert D\vert=\emptyset$, we set $r(D)=-1$. Otherwise we define $r(D)$ to the maximal $r\geq 0$ such that $\vert D-E\vert\neq \emptyset$ for all effective divisors $E$ of degree $r$. Now let $\Gamma$ be a disconnected curve, and let  $\Gamma=\Gamma_1\sqcup\cdots\sqcup\Gamma_k$ be the decomposition into connected components. We recall that for a divisor $D$ on an algebraic curve $X$ we have $r(D)=h^0(\calO_X(D))-1$, and that the quantity $h^0$ is additive in connected components. Motivated by this, we define $r_\Ga(D)$ by the identity \begin{equation*}
r_\Gamma(D)=\sum_{i=1}^k (r_{\Gamma_i}(D\vert_{\Gamma_i})+1)-1=\sum_{i=1}^k r_{\Gamma_i}(D\vert_{\Gamma_i})+k-1 \ .
\end{equation*}
The \emph{canonical divisor} $K_\Gamma$ on a metric graph $\Gamma$ is defined to be 
\begin{equation*}
    K_\Gamma=\sum_{p\in\Gamma} (\val(p)-2)\cdot p \ , 
\end{equation*}
where $\val(p)$ denotes the valence of $\Gamma$ at $p$, i.e. the number of outgoing half-edges from $p$.

The Riemann--Roch formula from~\cite{BakerNorine, MikhalkinZharkov, GathmannKerber} states that, for every divisor $D$ of degree $d$ on a compact and connected metric graph $\Gamma$ of genus $g$, we have 
\begin{equation*}
    r_\Gamma(D)-r_\Gamma(K_\Gamma-D)=d-g+1 \ .
\end{equation*}
As its classical counterpart, this combinatorial Riemann--Roch formula has already found numerous applications; we refer the reader to~\cite[Section 4.3]{BakerJensen} for a concise survey of these applications as well as an outline of its history and proof.

We first observe that the following extension of the Riemann--Roch formula holds for the Baker--Norine rank on arbitrary (possibly disconnected) compact metric graphs.

\begin{proposition}\label{prop_nonconnectedRR}
Let $\Gamma$ be a compact metric graph, and let $D$ be a divisor of degree $d$ on $\Gamma$. Then
\begin{equation*}
    r_\Gamma(D)-r_\Gamma(K_\Gamma-D)=d+\chi(\Gamma) \ .
\end{equation*}
\end{proposition}

\begin{proof}
    Let $\Gamma=\Gamma_1\sqcup\cdots\sqcup\Gamma_k$ be the decomposition into connected components and write $D_i=D|_{\Ga_i}$ for the restriction of $D$ to $\Gamma_i$ as well as $g_i$ for the genus of each $\Gamma_i$. Then $K_{\Ga_i}=K_{\Ga}|_{\Ga_i}$, and by the definition of $r_{\Ga}$ we have 
        \begin{equation*}\begin{split}
        r_\Gamma(D)-r_\Gamma(K_\Gamma-D)& =
        \sum_{i=1}^k r_{\Ga_i}\big(D\vert_{\Ga_i}\big)+k-1-\left[\sum_{i=1}^k r_{\Ga_i}\big(K_{\Gamma}\vert_{\Ga_i}-D\vert_{\Ga_i}\big)+k-1\right]\\
        & = \sum_{i=1}^k \big(r_{\Gamma_i}(D_i)-r_{\Gamma_i}(K_{\Gamma_i}-D_i)\big)\\
        &=\sum_{i=1}^k\big(\deg(D_i) - g_i+1\big)\\
        &=d +\chi(\Gamma) \ ,
    \end{split}\end{equation*}
since 
\begin{equation*}
\chi(\Gamma)=\sum_{i=1}^k\chi(\Gamma_i)=\sum_{i=1}^k(1-g_i) \ .
\end{equation*}
\end{proof}

\subsection{The higher rank case} Multidivisors allow us to generalize the definition of the Baker--Norine rank of a complete linear system on $\Gamma$ to tropical vector bundles. 

\begin{definition}
Let $\Gamma$ be a compact metric graph. We define the \emph{Baker--Norine rank} $r_\Gamma(f,D)$ of a multidivisor $(f\colon\tGa\to \Ga,D)$ simply to be the Baker--Norine rank $r_{\tGa}(D)$ of the divisor $D$ on $\widetilde{\Gamma}$. Given a vector bundle $E$ on $\Gamma$, we define its \emph{Baker--Norine rank} by
 \begin{equation*}
     r_\Gamma(E)=r_\Gamma(f,D)
 \end{equation*}
 for one (and automatically all) multidivisors $(f,D)$ such that $f_\ast \calH(D)\simeq E$.
\end{definition}

\begin{theorem}[Weil--Riemann--Roch]\label{thm_WeilRiemannRoch}
Let $\Gamma$ be a compact metric graph. For a vector bundle $E$ of degree $d$ on $\Gamma$ the formula
\begin{equation*}
r_\Gamma(E)-r_\Gamma\big(\calH(K_\Gamma)\otimes E^\ast\big)= d+n\chi(\Gamma) 
\end{equation*}
holds.
\end{theorem}

\begin{proof}
Let $(f,D)$ be a multidivisor such that $f_\ast\calH(D)\simeq E$. By Proposition~\ref{prop_multidivisors}, the multidivisor $(f,K_{\widetilde{\Gamma}}-D)$ represents the vector bundle $\calH(K_\Gamma)\otimes E^\ast$ on $\Gamma$. Then the Riemann--Roch formula for divisors from~\cite{BakerNorine, MikhalkinZharkov, GathmannKerber} (generalized to disconnected metric graphs as in Proposition~\ref{prop_nonconnectedRR}) can be applied to $D$ on $\widetilde{\Gamma}$ and tells us that
\begin{equation*}
    r_{\tGa}(D)-r_{\tGa}(K_{\widetilde{\Gamma}}-D)=d+\chi(\widetilde\Gamma) \ .
\end{equation*}
Since $f\colon \widetilde{\Gamma}\rightarrow \Gamma$ is a covering space of degree $n$, we have \begin{equation*}
    \chi(\widetilde\Gamma)=n\chi(\Gamma) \ .
\end{equation*}
Thus we obtain the desired formula
\begin{equation*}
r_\Gamma(E)-r_\Gamma\big(\calH(K_\Gamma)\otimes E^\ast\big)= d+n\chi(\Gamma) \ .
\end{equation*}
\end{proof}



\section{Semistable bundles and a Narasimhan--Seshadri correspondence}\label{section_NarasimhanSeshadri}

In this section we prove an analogue of the Narasimhan--Seshadri correspondence on a metric graph.

\begin{definition}
Let $\Gamma$ be a compact metric graph.
\begin{enumerate}[(i)]
    \item The \emph{slope} of a vector bundle $E$ on $\Gamma$ is defined to be the quotient 
    \begin{equation*}
        \mu(E)=\frac{\deg E}{\rk E}
    \end{equation*}
    \item A vector bundle $E$ on $\Gamma$ is said to be \emph{semistable} if for all subbundles $E'\subseteq E$, we have 
    \begin{equation*}
        \mu(E')\leq \mu(E) \ . 
    \end{equation*}
    If this inequality is strict for all proper subbundles, we say that $E$ is \emph{stable}.
\end{enumerate}
 
\end{definition}

\begin{proposition}\label{prop_stability}
Let $\Gamma$ be a compact and connected metric graph. 
\begin{enumerate}[(i)]
    \item A vector bundle on $\Gamma$ is stable if and only if it is indecomposable.
    \item A vector bundle $\calE$ that decomposes as $E=E_1\oplus \cdots \oplus E_r$ with all $E_i$ indecomposable is semistable if and only if 
    \begin{equation*}
        \mu(E_1)=\cdots =\mu(E_r) \ .
    \end{equation*}
\end{enumerate}
\end{proposition}

\begin{proof}
Let $E=E_1\oplus E_2$ for two vector bundles $E_1$ and $E_2$ on $\Gamma$ of ranks $n_1$ and $n_2$. Assume that $\mu(E_1)\leq \mu(E_2)$. Then we have 
\begin{equation*}
    \mu(E)=\frac{\deg(E_1)+\deg(E_2)}{n_1+n_2}=\frac{n_1}{n_1+n_2}\mu(E_1)+\frac{n_2}{n_1+n_2}\mu(E_2) \ .
\end{equation*}
This implies that
\begin{equation*}
    \mu(E_1)\leq \mu(E)\leq \mu(E_2) \ ,
\end{equation*}
and both inequalities are strict unless $\mu(E_1)=\mu(E_2)$. 
Since every sub-bundle splits by Proposition~\ref{prop_subbundlesplit} above this immediately implies Part (i). Part (ii) follows by induction over the number of indecomposable summands.
\end{proof}

\begin{example} \label{ex:g=1ss} Let $M_{n,d}(\Gamma)$ denote the locus of semistable bundles in $\Bun_n(\Gamma)$ of degree $d$. By Proposition \ref{prop_multidivisors}, a component of $M_{n,d}(\Gamma)$ parametrizes vector bundles associated to a disjoint union of connected multidivisors $(f_i,D_i)$ such that $\mu\big(\calH(f_i,D_i)\big)=\mu\big(\calH(f_j,D_j)\big)$ for all $i$ and $j$. Unlike its algebraic counterpart, the tropical moduli space $M_{n,d}(\Gamma)$ is, in general, not connected. 

Now suppose that $\Gamma$ is a circle, and denote $h=(n,d)$. There is a main component $M_{n,d}^\oplus(\Gamma)$ that parametrizes those semistable bundles of rank $n$ and degree $d$ that are direct sums of $h$ stable bundles (all of which must automatically be of rank $n/h$ and degree $d/h$ by Proposition~\ref{prop_stability} above). Combining Sections~\ref{section_tropicalmoduli} and~\ref{section_Atiyah}, we find that this component may be identified with the symmetric product $\Sym^h\Gamma$.

\end{example}


Let $\Ga$ be a metric graph. The sheaf $S_n\ltimes \mc H_{\Ga}^n$ has a natural subsheaf $S_n\ltimes \RR^n$, where $\RR\subset \mc H_{\Ga}$ is the sheaf of constant functions. We view sections of $S_n\ltimes \mc H_{\Ga}^n$ as $\GL_n(\TT)$-valued functions on $\Ga$. We now show that sections of $S_n\ltimes \RR^n$ can be viewed as functions valued in a tropical \emph{unitary} group, by proving an analogue of the Narasimhan--Seshadri correspondence. Indeed, consider a local system $\lambda$ on $\Gamma$ with fiber $S_n\ltimes\R^n$. The $S_n$-factor of $\lambda$ corresponds to a free degree $n$ cover $f:\tGa\to \Ga$, and there exists a local system $\widetilde{\lambda}$ on $\tGa$ with fiber $\R$ such that $f_\ast \widetilde{\lambda}=\lambda$. At the same time, $\lambda$ also naturally defines a rank-$n$ vector bundle $E(\lambda)$ on $\Ga$ with constant transition functions.  

\begin{theorem}[Narasimhan--Seshadri correspondence]\label{thm_NarasimhanSeshadri}
Let $\Gamma$ be a compact and connected metric graph. 
\begin{enumerate}[(i)]
    \item  A vector bundle $E$ of rank $n$ on $\Gamma$ is associated to an $S_n\ltimes\R^n$-local system $\lambda$ if and only if it is semistable and of degree zero. The vector bundle $E(\lambda)$ is stable if and only if the corresponding $S_n$-representation of $\pi_1(\Ga)$ is indecomposable.
    \item Two $S_n\ltimes\R^n$-local systems $\lambda_i$ (for $i=1,2$) give rise to the same vector bundle if and only if they define the same cover $f\colon\widetilde \Gamma\rightarrow\Gamma$ and the induced classes of the $\widetilde{\lambda}_i$ in $\Jac(\widetilde{\Gamma})$ are equal.
\end{enumerate}

\end{theorem}

\begin{proof} We first consider the case $n=1$, where we need to show that any line bundle of degree zero (which is trivially semistable) is associated to an $\R$-local system, corresponding to an element $H^1(\Gamma,\R)$. Consider the long exact sequence~\eqref{eq:les}, written in \v{C}ech cohomology with respect to an appropriate cover. The set of isomorphism classes of line bundles is $\Pic(\Ga)=\check{H}^1(\Ga,\mc H_{\Ga})$. The subset of degree zero line bundles is the image of $\check{H}^1(\Ga,\RR)$. Hence any line bundle of degree zero can be represented by an $\RR$-valued cocycle. Two $\R$-local systems $\lambda_i$ (for $i=1,2$) give rise to the same line bundle if and only if their classes in the quotient $\Jac(\Ga)=\check{H}^1(\Ga,\RR)/\Omega_{\Ga}(\Ga)$ agree.

We now consider the general case. By Proposition~\ref{prop_stability}, a vector bundle $E$ of degree (and hence slope) zero is semistable if and only if it is a direct sum of stable vector bundles of degree zero. By Proposition~\ref{prop_multidivisors}, this is equivalent to saying that $E$ is the pushforward $f_*L$, along a degree $n$ free cover $f:\tGa\to \Ga$, of a line bundle $L$ having degree zero on each connected component of $\tGa$. Hence if $E$ is semistable, then we can choose the transition functions of $L$ on each connected component to be constant, and then the transition functions of $E$ are also constant by the definition of pushforward. Conversely, if $E$ is associated to a $S_n\ltimes\R^n$-local system, then the transition functions of $L$ can be chosen to be constant, so $L$ has degree zero on each connected component and therefore $E$ is semistable. 

By Proposition \ref{prop_stability} the vector bundle $E(\lambda)$ is stable if and only if $\widetilde{\Gamma}$ is connected, which is the case if and only if the representation $\pi_1(\Gamma)\rightarrow S_n$ giving rise to $f\colon\widetilde{\Gamma}\rightarrow \Gamma$ is indecomposable.
\end{proof}

\begin{remark}
We point out that every representation $\pi_1(\Gamma)\rightarrow S_n\ltimes \R^n$ is already a direct sum of irreducible representations. This is why we do not have to require the representation in Theorem~\ref{thm_NarasimhanSeshadri} to be ``semisimple'' in any suitable sense. 
\end{remark}



\section{The process of tropicalization}\label{section_tropicalization}

In this section and the next, we work over an algebraically closed field $K$ that is complete with respect to a non-trivial non-Archimedean absolute value $\vert.\vert$. Write $\val$ for the associated valuation on $K$, as well as $R$ for the valuation ring of $K$ and $k$ for its (automatically algebraically closed) residue field.


\subsection{Freely $S_n\ltimes\G_m^n$-linearized vector bundles}\label{section_linearizedvectorbundles}

Let $X$ be a \emph{Mumford curve}, i.e.\ a (connected) smooth and projective curve over $K$ that admits a prestable reduction $\calX$ whose special fiber is a nodal curve, all of whose components normalize to $\PP^1$. A vector bundle $E$ on $X$ is said to be $S_n\ltimes \G_m^n$-\emph{linearized} if the structure group of $E$ is $S_n\ltimes\G_m^n$; this means that we may trivialize $E$ on an open cover $U_i$ of $X$ such that the transition maps, which a priori live in $\GL_n\big(\calO(U_i\cap U_j)\big)$, may be chosen to land in $S_n\ltimes\G_m^n\subseteq \GL_n$. 

A vector bundle is $S_n\ltimes \G_m^n$-linearized if and only if there is an \'etale degree $n$ cover $f\colon \widetilde{X}\rightarrow X$ (with $\widetilde{X}$ possibly disconnected) and a line bundle $L$ on $\widetilde{X}$ such that $f_\ast L\simeq E$. We say that $E$ is \emph{freely  $S_n\ltimes\G_m^n$-linearized} if the cover $f\colon\widetilde{X}\rightarrow X$ is induced from an \'etale  cover $\widetilde{\calX}\rightarrow \calX$ of the prestable model $\calX$ of $X$ (and not from an admissible cover in the sense of~\cite{HarrisMumford}). Note that this already implies that $\widetilde{X}$ is also a Mumford curve. Being freely linearized is equivalent to the condition that the induced map on the dual metric graphs (respectively on the skeletons) is a free cover (in general, this map is a harmonic morphism with nontrivial dilation). Equivalently, we may require the induced map $f^{\an}\colon \widetilde{X}^{\an}\rightarrow X^{\an}$ of Berkovich analytic spaces to be a topological covering space. Hence this condition is also independent of the choice of the particular prestable model $\calX$ of $X$. 

Write $\Bun_{S_n\ltimes\G_m^n}(X)$ for the moduli stack of vector bundles of rank $n$, which are $S_n\ltimes\G_m^n$-linearized. Arguing as in Corollary~\ref{cor_modulicomponents}, we find:

\begin{lemma}
The moduli stack $\Bun_{S_n\ltimes\G_m^n}(X)$ is given by the disjoint union
\begin{equation*}
    \Bun_{S_n\ltimes \G_m^n}(X)=\bigsqcup_{f\colon \widetilde{X}\rightarrow X} \big[\Bun_{\G_m}(\widetilde{X})\big/\Aut(f)\big]
\end{equation*}
taken over isomorphism classes of \'etale covers $f\colon\widetilde{X}\rightarrow X$ of degree $n$, where $\Aut(f)$ is the group of deck transformations of $f$ acting on $\Bun_{\G_m}(\widetilde{X})$ via pullback. 
\end{lemma}

For fixed $f\colon\widetilde{X}\rightarrow X$ we write 
\begin{equation*}
    \widetilde{X}=\widetilde{X}_1\sqcup \cdots \sqcup \widetilde{X}_k
\end{equation*} for the decomposition into connected components. Then we have
\begin{equation*}
    \Bun_{\G_m}(X)= \Bun_{\G_m}(\widetilde{X}_1)\times \cdots \times \Bun_{\G_m}(\widetilde{X}_k) \ .
\end{equation*}
Rigidifying by $\G_m$-automorphism groups, we find \begin{equation*}
\big[\Bun_{\G_m}(\widetilde{X}_i)\fatslash\ \mathbf{B}\G_m\big]\simeq \Pic(\widetilde{X}_i)
\end{equation*}
and every $\Pic_d(\widetilde{X}_i)$ is naturally a torsor over the Jacobian $\Pic_0(\widetilde{X}_i)\simeq\Jac(\widetilde{X}_i)$.

The stack $\Bun_{S_n\ltimes \G_m^n}(X)$ is a smooth and proper Artin stack that arises as a disjoint union of finite group quotients of the $\Bun_{\G_m}(\widetilde{X}_i)$. The subspace $\Bun^{\free}_{S_n\ltimes\G_m^n}(X)$ of freely $S_n\ltimes \G_m^n$-linearized vector bundles consists precisely of those components, for which the cover $f\colon \widetilde{X}\rightarrow X$ is induces a topological covering space $f^{\an}\colon \widetilde{X}^{\an}\rightarrow X^{\an}$. 

\subsection{Tropicalization}
Let $X$ be a Mumford curve and write $\Gamma_X$ for a \emph{non-Archimedean skeleton} of the Berkovich analytic space $X^{\an}$ that is associated to a prestable model $\calX$ of $X$ over $R$ (e.g.\ the minimal semistable model when $g(X)\geq 2$). Write $\Bun_{S_n\ltimes \G_m^n}^{\free}(X)^{\an}$ for the Berkovich analytification of $\Bun_{S_n\ltimes \G_m^n}^{\free}(X)$ and, implicitly, also for its underlying topological space (see~\cite[Section 3]{Ulirsch_tropisquot} for details). 

Then there is a natural \emph{tropicalization map}
\begin{equation*}
    \trop\colon \Bun_{S_n\ltimes \G_m^n}^{\free}(X)^{\an}\longrightarrow \Bun_n(\Gamma_X)
\end{equation*}
given by the following: A point in $\Bun_{S_n\ltimes \G_m^n}^{\free}(X)^{\an}$ may be represented by an $K'$-valued point of $\Bun_{S_n\ltimes \G_m^n}^{\free}(X)$ for a non-Archimedean extension $K'$ of $K$. This datum corresponds to a freely $S_n\ltimes\G_m^n$-linearized vector bundle $E$ on the base change $X_{K'}$ of $X$ to $K'$. Let $f\colon\widetilde{X}\rightarrow X_{K'}$ be an \'etale cover of degree $n$ and $L$ a line bundle on $\widetilde{X}$ such that $f_\ast L=E$. The cover $f$ is the generic fiber of a degree $n$ \'etale cover $\widetilde{\calX}\rightarrow \calX$ of prestable models over $S$, where $S$ is the valuation ring of $K'$; thus the induced map $f^{\trop}\colon\Gamma_{\widetilde{X}}\rightarrow \Gamma_X$ is a free cover of degree $n$. Then we set
\begin{equation*}
    \trop(E)=\big(f^{\trop},\trop(L)\big)
\end{equation*}
where $\trop(L)$ is the usual tropicalization of a line bundle $L$ on $\widetilde{X}_L$ thought of as a $K'$-valued point of $\Pic(X)^{\an}$. 

To describe $\trop(L)$ we expand on~\cite[Section 6.3]{BakerJensen}: We may assume that $K'$ is algebraically closed. Write $L=\calO_{X_{K'}}(D)$ for a split divisor $$D=\sum_{p\in X(K')}D(p)\cdot p$$ on $\widetilde{X}_{K'}$. Then $\trop(L)$ is defined to be the linear equivalence class of the pointwise pushforward of $D$ to the skeleton $\Gamma_{\widetilde{X}}=\Gamma_{\widetilde{X}_{K'}}$ of $\widetilde{X}^{\an}$. This is well-defined thanks to the slope formula from~\cite{BPR} (also see~\cite{BPR_survey}).

There is a second way to associate to $\Bun_{S_n\ltimes \G_m^n}^{\free}(X)$ a disjoint union of real torus torsors. For this note first that the underlying topological space of $\Bun_{S_n\ltimes \G_m^n}^{\free}(X)^{\an}$ is naturally homeomorphic to the underlying topological space of 
\begin{equation}\label{eq_Picquot}
    \bigsqcup_{f\colon \widetilde{X}\rightarrow X} \big[\Pic(\widetilde{X})\big/\Aut(f)\big]^{\an} \ ,
\end{equation} 
since rigidifying by $\mathbf{B}\G_m^{\an}$ does not change the points of a non-Archimedean stack. Every component in \eqref{eq_Picquot} is a finite quotient of torsors over an abelian variety and, using Raynaud's uniformization, one can show that the Berkovich analytification of each component admits a strong deformation retraction onto a non-Archimedean skeleton that is a torsor over a real torus (see~\cite[Section 6.5]{Berkovich_book} and~\cite[Section 4]{BakerRabinoff}). We write
\begin{equation*}
    \rho\colon \Bun_{S_n\ltimes\G_m^n}^{\free}(X)^{\an}\longrightarrow \Sigma\big(\Bun_{S_n\ltimes \G_m^n}^{\free}(X)\big)
\end{equation*}
for the disjoint union of these retractions. 

The tropicalization map agrees with $\rho$ in the following sense:

\begin{theorem}\label{thm_Buntrop=skel}
There is a natural isomorphism $J\colon \Bun_n(\Gamma_X)\xrightarrow{\sim} \Sigma\big(\Bun_{S_n\ltimes \G_m^n}^{\free}(X)\big)$ that makes the diagram 
\begin{equation*}
    \begin{tikzcd}
        & \Bun_{S_n\ltimes \G_m^n}^{\free}(X)^{\an} \arrow[ld,"\trop"'] \arrow[rd,"\rho"] &\\ 
        \Bun_n(\Gamma_X) \arrow[rr,"J"',"\sim"] & & \Sigma\big(\Bun_{S_n\ltimes \G_m^n}^{\free}(X)\big)
    \end{tikzcd}
\end{equation*}
commute. 
\end{theorem}

Theorem~\ref{thm_Buntrop=skel}, in particular, tells us that the tropicalization map is well-defined, continuous, proper, and surjective, since these are well-known properties of the retraction map to the skeleton. 

\begin{proof}[Proof of Theorem~\ref{thm_Buntrop=skel}]
The case $n=1$ is an immediate consequence of the main result of~\cite{BakerRabinoff}.
We have $\Pic(X)=\bigsqcup_{d\in\Z}\Pic_d(X)$ as well as $\Pic(\Gamma_X)=\bigsqcup_{d\in\Z}\Pic_d(\Gamma_X)$ and, since tropicalization preserves degree, we may consider each component separately.

Choose a base point $p\in X(K)$ and write $\alpha\colon X\rightarrow\Pic_0(X)\simeq\Jac(X)$ as well as $\alpha^{\trop}\colon \Gamma_X\rightarrow \Jac(\Gamma_X)$ for the associated Abel--Jacobi maps, both given by $q\mapsto [q-p]$. By~\cite{BakerRabinoff} we have a natural isomorphism $\Jac(\Gamma_X)\xrightarrow{\sim} \Sigma(\Jac(X))$ between the Jacobian $\Jac(\Gamma_X)$ and the non-Archimedean skeleton of $\Jac(X)^{\an}$ that makes the diagram 
\begin{equation*}
    \begin{tikzcd}
        X^{\an}\arrow[d,"\rho'"]\arrow[rr,"\alpha^{\an}"] &&\Jac(X)^{\an}\arrow[d,"\rho"]\\
        \Gamma_X\arrow[r,"\alpha^{\trop}"]&\Jac(\Gamma_X)\arrow[r,"\sim"]& \Sigma(\Jac(X))
    \end{tikzcd}    
\end{equation*}
commute. This immediately implies the commutativity of
\begin{equation*}
    \begin{tikzcd}
        \Pic_0(X)^{\an}\arrow[rr,"\sim"]\arrow[d,"\trop"']&&\Jac(X)^{\an}\arrow[d,"\rho"]\\
        \Pic_0(\Gamma_X)\arrow[r,"\sim"]&\Jac(\Gamma_X)\arrow[r,"\sim"]&\Sigma(\Jac(X))
    \end{tikzcd}
\end{equation*}
and thus the claim for the component $\Pic_0(X)$.

Using the base point $p\in X(K)$ and its image $q=\rho(p)\in\Gamma_X$, we consider the isomorphisms 
\begin{equation*}\begin{split}
 \Pic_0(X)&\xlongrightarrow{\sim}\Pic_d(X)\\
[D]&\longmapsto [D+dp]
\end{split}\end{equation*}
and 
\begin{equation*}\begin{split}
 \Pic_0(\Gamma_X)&\xlongrightarrow{\sim}\Pic_d(\Gamma_X)\\
[D]&\longmapsto [D+dq] \ .
\end{split}\end{equation*}
We obtain an isomorphism $J_d\colon \Pic_d(\Gamma)\xrightarrow{\sim}\Sigma(\Pic_d(X))$ that makes the diagram
\begin{equation*}
    \begin{tikzcd}
        \Pic_d(X)^{\an}\arrow[rr,"\sim"]\arrow[d,"\trop"']&&\Jac(X)^{\an}\arrow[d,"\rho"]\\
        \Pic_d(\Gamma_X)\arrow[rr,"\sim"]&&\Sigma\big(\Pic_d(X)\big)
    \end{tikzcd}
\end{equation*}
commute. This vertical arrows here do not depend on the choice of the base point $p$. Since 
\begin{equation*}
\big[\Bun_{\G_m}(X)\fatslash\  \mathbf{B}\G_m\big]\simeq \Pic(X)
\end{equation*} 
this is our claim for $n=1$. 

The general case $n\geq 1$ follows from the fact that there is a natural one-to-one correspondence between \'etale covers of $\calX$ and free covers of $\Gamma_X$ and that pullback along automorphisms commutes with the formation of skeleton and tropicalization. 
\end{proof}

\begin{remark}
Any free cover  $\widetilde{\Gamma}\rightarrow\Gamma_X$ of finite degree naturally arises as the skeleton of an algebraic cover that is induced by an \'etale cover on the level of prestable models, since the local Hurwitz numbers are all equal to one. See~\cite{ABBRI, ABBRII, CavalieriMarkwigRanganathan_tropadmissiblecovers} for background on this.
\end{remark}

\begin{proposition}\label{prop_tropbasicoperations}
Tropicalization is naturally compatible with direct sums, tensor products, dualization, determinants, and degrees of vector bundles.
\end{proposition}

\begin{proof}
This follows from the compatibility of tropicalization with the natural constructions for covers from Proposition~\ref{prop_multidivisors}.

To give an explicit example, we consider the case of the tensor product. Given two freely $S_n\ltimes\G_m^n$-linearized vector bundles $E_1$ and $E_2$, we write $E_i=(f_i)_\ast L_i$ (for $i=1,2$), where $f_i\colon \widetilde{X}_i\rightarrow X$ is an \'etale cover (induced from an \'etale cover $\widetilde{\calX}_i\rightarrow \calX$) and $L_i$ is a line bundle on $\widetilde{X}_i$. Write $L_i=\calO_{\widetilde{X}_i}(D_i)$ for a divisor $D_i$ on $\widetilde{X}_i$. Consider the fibered product $f_1\times_X f_2\colon \widetilde{X}_1\otimes_X\widetilde{X}_2\rightarrow X$ of the two covers as well as the divisor $D_1\boxtimes D_2$ on $\widetilde{X}_1\times_X\widetilde{X}_2$ defined by 
\begin{equation*}
    (D_1\boxtimes D_2)(p,q)=D_1(p)+D_2(q) \ .
\end{equation*}
Then we have $(f_1\times_X f_2)_\ast \calO_{\widetilde{X}_1\times_X\widetilde{X}_2}(D_1\boxtimes D_2)\simeq E_1\otimes E_2$ and $f_1\times_X f_2$ is the generic fiber of the \'etale cover $\widetilde{\calX}_1\times_{\calX}\widetilde{\calX}_2\rightarrow \calX$ on the level of prestable models, which induces the free cover $f_1^{\trop}\times_{\Gamma_X} f_2^{\trop}\colon \Gamma_{\widetilde{X}_1}\times_{\Gamma_X}\Gamma_{\widetilde{X}_2}\rightarrow \Gamma_X$ on the level of dual graphs. The other cases follow with an analogous argument.
\end{proof}

The following is a higher rank generalization of Baker's specialization inequality from~\cite{Baker_specialization}.

\begin{proposition}\label{prop_specializationinequality}
Let $X$ be a Mumford curve over $K$ and $\Gamma_X$ a skeleton of $X^{an}$. Then for every freely $S_n\ltimes\G_m^n$-linearized vector bundle $E$ on $X$ we have 
\begin{equation*}
    h_X^0(E)\leq r_\Gamma\big(\trop_X(E)\big)+1 \ .
\end{equation*}
\end{proposition}

\begin{proof} Let $f\colon\widetilde{X}\rightarrow X$ be an \'etale cover induced from an \'etale cover of prestable models and $L$ a line bundle on $\widetilde{X}$ such that $f_\ast L=E$. Write $\widetilde{X}=\widetilde{X}_1\sqcup \cdots \sqcup \widetilde{X}_k$ for its decomposition into connected components and $\widetilde{\Gamma}=\widetilde{\Gamma}_1\sqcup \cdots \sqcup \widetilde{\Gamma}_k$ for the corresponding decomposition of dual metric graphs. Then we have
\begin{equation*}
    h^0(E)=h^0(L)=\sum_{i=1}^k h^0\big(L\vert_{\widetilde{X}_i}\big) \leq \sum_{i=1}^k \big(r_{\widetilde{\Gamma}_i}(\trop(L\vert_{\widetilde{X}_i}))+1\big)=r_\Gamma\big(\trop(E)\big)+1
\end{equation*}
by  Baker's specialization inequality~\cite[Lemma 2.8 and Cor.\ 2.10]{Baker_specialization} applied to each line bundle $L\vert_{\widetilde{X}_i}$.
\end{proof}


\section{The case of a Tate curve}\label{section_Tatecurve}

From now on let $K$ be also of characteristic zero. In this section we consider the special case of a \emph{Tate curve}, i.e. of a (connected) smooth and projective curve $X$ over $K$ of genus one, for which $\val(j(X))< 0$. In this case, we may write $X^{\an}$ as a quotient 
\begin{equation*}
    X^{\an}=\G_m^{\an}/q^\Z
\end{equation*}
for a unique $q\in K^\ast$ with $\val(q)=-\val(j)$. In particular, the Tate curve $X$ admits a semistable model $\calX$ over $R$, whose special fiber is a circle of projective lines. The (minimal) non-Archimedean skeleton $\Gamma_X$ of $X^{\an}$ is a circle of length $\val(q)$, as considered in Section~\ref{section_Atiyah} above. 

\subsection{Vector bundles on elliptic curves} 
In~\cite{Atiyah_ellipticvectorbundles} Atiyah provides us with a classification of all indecomposable vector bundles of rank $n$ and degree $d$ on an elliptic curve $X$. For this we assume from now on that $\characteristic(K)=0$.

Atiyah first constructs for each $h\geq 1$ a unique indecomposable vector bundle $F_h$ of rank $h$ and degree $0$ such that $h^0(X,E)=1$ (by the Riemann--Roch theorem, $h^0(X,E)=0$ or $h^0(X,E)=1$ are the only possibilities). The line bundle $F_1$ is the trivial line bundle on $E$ and, inductively, for $h\geq 2$ the vector bundle $F_h$ is the unique nontrivial extension 
\begin{equation*}
    0\longrightarrow \calO_X\longrightarrow F_h\longrightarrow F_{h-1}\longrightarrow 0 \ . 
\end{equation*}
Now fix a base point $p$ in $X$ and let $h=(n,d)$ as well as $n'=\frac{n}{h}$ and $d'=\frac{d}{h}$. Atiyah constructs a canonical indecomposable vector bundle $E(n,d)$ of rank $n$ and degree $d$ for every $n\geq 1$ and $d\in\Z$ (depending only on the chosen base point) such that every indecomposable vector bundle $E$ of rank $n$ and degree $d$ on $X$ is of the form 
\begin{equation*}
    E=E(n,d)\otimes L
\end{equation*}
for a line bundle $L\in\Pic_0(X)$. The line bundle $L$ is unique up to multiplication by $n'$-torsion elements. In degree zero we have $E(h,0)=F_h$ and the compatibilities
\begin{equation*}
    E(n,d)\simeq E(n',d')\otimes F_h
\end{equation*}
as well as 
\begin{equation*}
    E(n,d+n)\simeq E(n,d)\otimes \calO_X(p)
\end{equation*}
hold, where $\calO_X(p)$ denotes the line bundle associated to the chosen base point $p\in X$.

We now rephrase Atiyah's classification in the case of the Tate curve using ideas of Oda~\cite{Oda_elliptic} (see in particular~\cite[Theorem 2.18]{BBDG} for a complex-analytic version of this).

For this we need cyclic covers of $X$ that are themselves Tate curves. There are two ways to understand such  covers, one using analytic uniformization and the other tropicalizations of torsion line bundles. For the analytic point of view, we pick an isomorphism $X^\an= \G_m^\an /q^ \Z$, where $q\in K^*$ is uniquely determined. Then the quotient morphism $\G_m^\an/q^{n\Z}\to \G_m^\an/q^\Z$ is the unique free connected degree $n$ cover of $X^{\an}$, which in turn induces a cover of Tate curves up to unique isomorphism. We denote this cover by $f\colon X_n\to X$. 

For the alternative construction using tropicalizations of line bundles, consider the tropicalization map on $n$-torsion points of the Jacobian
\[
\Z/n\Z\times \Z/n\Z\simeq \Jac_n(X)(K) \to \Jac_n(\Gamma)\simeq \Z/n\Z \ ,
\]
where $\Gamma$ is the tropicalization of $X$. By~\cite[Theorem 3.1]{JensenLen} this morphism is surjective and therefore its kernel is cyclic of order $n$. Let $L$ be a generator. After picking a suitable semistable model $\mc X$ of $X$, we can extend $L$ to a line bundle $\mc L$ on $\mc X$ of multidegree $0$. Then $\mc L^{\otimes n}$ is a line bundle of multidegree $0$ that extends the trivial bundle on $X$. Therefore, $\mc L^{\otimes n}$ is trivial. If $s\in \Gamma(\mc X, \mc L^{\otimes n})$ is a nowhere vanishing section, then 
\begin{equation*}
\calX_n=\calSpec \Big(\bigoplus_{k\geq 0} \calL^{\otimes k}\big/_{1\sim s} \Big)
\end{equation*}
is a cyclic $n$-cover of $\mc X$. In particular, the generic fiber of $\mc X_n$ is a cyclic $n$-cover of $X$ that is itself a Tate curve and therefore it is isomorphic to the cover $X_n$ from above. Note that different choices of $L$ produce the same underlying \'etale cover, but the $\Z/n\Z$-action differs by the action of $(\Z/n\Z)^*$. We will not actually need the $\Z/n\Z$-action and therefore the cover $f\colon X_n\to X$ is unique for our purposes.

\begin{theorem}\label{thm_Oda}
Let $E$ be an indecomposable vector bundle of rank $n$ and degree $d$ on a Tate curve $X$. Write $h=(n,d)$ as well as $n=n'\cdot h$ and $d=d'\cdot h$, and let $f:X_{n'}\to X$ be the unique connected free cover of degree $n'$. Then there exists a unique line bundle $L$ on  $X_{n'}$ such that 
\begin{equation*}
  E\simeq f_\ast L \otimes F_h \ .
\end{equation*}
\end{theorem}

\begin{proof}
Let $\widetilde{L}$ be a line bundle on $X_{n'}$ of degree $d'$. The proof of~\cite[Theorem 2.18]{BBDG}  word-for-word shows that $f_\ast \widetilde{L}$ is an indecomposable vector bundle on $X$ of rank $n'$ and degree $d'$. The central geometric point in this proof is that $X_{n'}$ is connected. 

Now Atiyah's classification allows us to find a line bundle $M\in\Pic_0(X)$ such that 
\begin{equation*}
    E\simeq f_\ast \widetilde{L}\otimes M \otimes F_h \ ,
\end{equation*}
and the projection formula tells us that
\begin{equation*}
   f_\ast \widetilde{L}\otimes M \simeq f_\ast \big(\widetilde{L}\otimes f^\ast M\big) \ .
\end{equation*}
The line bundle $M\in\Pic_0(X)$ is unique up to multiplication by an $n'$-torsion element, but passing to the pullback $f^\ast M$ avoids this ambiguity. Therefore $L:=\widetilde{L}\otimes f^\ast M$ is the unique line bundle such that $f_\ast L\otimes F_h\simeq E$. 
\end{proof}

\begin{proposition}\label{prop_tropE}
Let $X$ be a Tate curve and let $n,d$ be integers with $(n,d)=1$. Furthermore, let $f\colon X_n\to X$ be the natural \'etale cover of degree $n$, and let $p_0^{(n)}$ be the origin of $X_n$. Then we have
\[
E(n,d)\simeq f_*\big(\mc O(dp_0^{(n)}) \otimes M\big)  \ ,
\]
where $M$ is the unique order-$(n,2)$ point of $\Jac(X_n)$ with trivial tropicalization.
\end{proposition}

\begin{proof}
By~\cite[Corollary to Theorem 7]{Atiyah_ellipticvectorbundles}, two indecomposable bundles of rank $n$ and degree $d$ are isomorphic if and only if their determinants are, so it suffices to compute the determinants.
By~\cite[Theorem 6]{Atiyah_ellipticvectorbundles} we have $\det E(n,d)=\mc O(dp_0)$, where $p_0=f\big(p_0^{(n)}\big)$. On the other hand, we have 
\[
\det \big(f_*\mc O(dp_0^{(n)})\big)= \mc O(dp_0) \otimes \det(f_*\mc O_{X_n}) \ .
\]
By the discussion above, we have 
\[
f_*\mc O_{X_n}= \bigoplus_{k=0}^{n-1} L^{\otimes k}
\]
for some order-$n$ point of $\Jac(X)$ with trivial tropicalization and therefore $\det(f_*\mc O_{X_n})= \mc L^{\frac {n(n+1)}2}$, which has trivial tropicalization and order $(n,2)$. In particular, if $n$ is odd we are done. If $n$ is even, let $P$ be an element of order $2n$ of $\Jac(X)$ with trivial tropicalization. Then $P^{\otimes n}$ has order $2$ and trivial tropicalization, hence $P^{\otimes n}= \det(f_*\mc O_{X_n})$.  Let $M=f^*P$. As $P^{\otimes 2}$ is a power of $L$ and $\ker(f^*)=\langle L\rangle$, the line bundle $M$ has order $2$ and trivial tropicalization. Furthermore, we have
\[
\det f_*\big(\mc O(dp_0^{(n)}) \otimes M\big)=  \det \big(f_*\mc O(dp_0^{(n)})\otimes P\big)= \det \big(f_*\mc O(dp_0^{(n)})\big) \otimes P^{\otimes n}= \mc O(dp_0) \ ,
\]
finishing the proof.
\end{proof}



\subsection{Moduli spaces of semistable bundles}
Given a vector bundle $E$ of rank $n$ and degree $d$ on a smooth projective algebraic curve $X$, we write $\mu(E)=\frac{d}{n}$ for its \emph{slope}. Recall that the vector bundle $E$ is said to be \emph{semistable} if  the inequality $\mu(E')\leq \mu(E)$ holds for all subbundles $E'$ of $E$; it is called \emph{stable} if the inequality is strict for all proper non-zero subbundles.  For a vector bundle of slope $\mu$ there always is a natural \emph{Jordan--H\"older filtration}
\begin{equation*}
    0=V_0\subseteq V_1\subseteq \cdots\subseteq V_s=E
\end{equation*}
by subbundles such that the quotients $V_i/V_{i-1}$, the so-called \emph{Jordan--H\"older factors}, are  stable vector bundles of slope $\mu$ for all $i=1,\ldots, s$. We associate to $E$ its \emph{graded vector bundle} 
\begin{equation*}
    \gr(E)=\bigoplus_{i=1}^s V_i/V_{i-1} 
\end{equation*}
and say that two semistable vector bundles $E$ and $E'$ are \emph{equivalent} if $\gr(E)\simeq \gr(E')$. We write $M_{n,d}(X)$ for the moduli space of semistable vector bundles of rank $n$ and degree $d$ on $X$, whose $K$-points are in one-to-one correspondence with equivalence classes of semistable vector bundles.

Let now $X$ be of genus one. It is well-known that Atiyah's classification implies that an indecomposable vector bundle of rank $n$ and degree $d$ on $X$ is always semistable, and that it is stable if and only $n$ and $d$ are coprime. We refer the reader to~\cite[Appendix A]{Tu_semistablebundles} for a proof of these facts. 

Write again $h=(n,d)$ as well as $n'=\frac{n}{h}$ and $d'=\frac{d}{h}$. In~\cite[Theorem 1]{Tu_semistablebundles} Tu proves that there is a natural isomorphism 
\begin{equation*}
    M_{n,d}(X)\simeq \Sym^h X
\end{equation*}   
with the $h$-th symmetric power $\Sym^h X=X^h/S_h$ of $X$. The isomorphism is given as follows. A semistable vector bundle $E\in M_{n,d}(X)$ is equivalent to its graded vector bundle $\gr(E)$. Thus we may assume that $E$ is a direct sum 
\begin{equation*}
    E=\bigoplus_{i=1}^s E_i
\end{equation*}
of stable bundles $E_i$ on $X$. By an argument analogous to the one in Proposition~\ref{prop_stability} (ii), we must have $\mu(E_i)=\mu(E_j)$ for all $i=1,\ldots, s$. Moreover, since stable bundles are precisely the indecomposable bundles of coprime rank and degree, we must have $s=h$ and $\mu(E_i)=\frac{d'}{n'}$ for all $i=1,\ldots, s$. By Atiyah's classification each $E_i$ is of the form 
\begin{equation*}
    E_i=E(n',d')\otimes M_i
\end{equation*}
for a line bundle $M_i\in\Pic_0(X)$ that is unique up to multiplication by an $n'$-torsion element. The isomorphism $M_{n,d}(X)\xrightarrow{\sim} \Sym^h \Pic_1(X)\simeq \Sym^h X$ is then given by the association
\begin{equation*}
    E=\bigoplus_{i=1}^h E(n',d')\otimes M_i\longmapsto  \sum_{i=1}^h M_i^{n'} \otimes\calO_X(p)\ ,
\end{equation*}
where we canonically identify $X$ with $\Pic_1(X)$.

\subsection{Tropicalization}\label{section_Tatecurvetropicalization} Let $X$ be a Tate curve with $X^{\an}=\G_m^{\an}/q^\Z$ and $\Gamma_X$ its (minimal) non-Archimedean skeleton, which is a circle of length $\val(q)$. Let $n\geq 1$ and $d\in\Z$. Write again $h=(n,d)$ as well as $n'=\frac{n}{h}$ and $d'=\frac{d}{h}$. The moduli space of stable bundles on $\Ga_X$ is described in Section~\ref{section_Atiyah} above. 

There is a natural tropicalization map
\begin{equation*}
    \trop\colon M_{n,d}(X)^{\an}\longrightarrow \Bun_n(\Gamma_X)
\end{equation*}
defined as follows: 
\begin{itemize}
    \item A point in $M_{n,d}(X)^{\an}$ is represented by a semistable vector bundle $E$ of rank $n$ and degree $d$ on $X_{K'}$ for a non-Archimedean extension $K'$ of $K$; to avoid unnecessary notation we may assume that $K'$ is algebraically closed. Then, as above, the vector bundle $E$ is equivalent to a direct sum
    \begin{equation*}
        E=\bigoplus_{i=1}^h E_i
    \end{equation*}
    where each $E_i$ is a stable vector bundle of rank $n'$ and degree $d'$ on $X_{K'}$.
    \item By Theorem~\ref{thm_Oda}, we find  unique line bundles $L_i$ on $(X_{K'})_{n'}$ such that $E_i=f_\ast L_i$. We can define each $\trop(E_i)$ as the pushforward $f_\ast^{\trop} \trop(L_i)$, as in Section~\ref{section_tropicalization} above. The tropicalization $\trop(E)$ of $E$ is given by 
    \begin{equation*}
        \trop(E)=\bigoplus_{i=1}^h f_\ast^{\trop} \trop(L_i) \ .
    \end{equation*}
\end{itemize}

 We begin by stating the following immediate consequence of Proposition~\ref{prop_tropE} above that essentially tells us that Atiyah's classification is compatible with tropicalization when $(n,d)=1$. 

\begin{proposition}\label{prop_Etrop}
\label{lem_tropicalization of E(n,d)}
Let $n,d$ be integers with $(n,d)=1$. Then we have
\[
\trop\big(E(n,d)\big)=E^\trop(n,d) \ ,
\]
where $E^\trop(n,d)$ is the vector bundle on $\Ga_X$ defined in Section~\ref{section_Atiyah}.
\end{proposition}

\begin{proof}
By Proposition~\ref{prop_tropE}, we have 
\[
E(n,d)= f_*\big(\mc O(dp_0^{(n)})\otimes M\big) \ ,
\]
 where $f\colon X_n\to X$ is the covering map from above, $p_0^{(n)}$ is the identity element on $X_n$, and $M$ is a line bundle with trivial tropicalization. By the definitions of the tropicalization and of $E^\trop(n,d)$, it follows that
 \[
 \trop\big(E(n,d)\big)= f^\trop_*\trop\big(\mc O(dp_0^{(n)})\otimes M\big)= f^\trop_*\trop\big(\mc O(dp_0^{(n)})\big) = E^\trop(n,d) \ .
 \]
\end{proof}

\begin{remark}
Note that the statement of Proposition~\ref{lem_tropicalization of E(n,d)} is not true if $n$ and $d$ are not coprime. In fact, if $(n,d)=h$ and $n'$ and $d'$ are defined by $n'h=n$ and $d'h=d$, then $\trop\big(E(n,d)\big)=E^\trop(n',d')^h$. The reason for this is that $E^\trop(n,d)$ is stable tropically, while $E(n,d)$ is not stable algebraically. 
\end{remark}

%

Recall now that $M_{n,d}^\oplus(\Gamma_X)$ denotes the component of $M_{n,d}(\Gamma_X)$ that parametrizes semistable bundles of rank $n$ and degree $d$ that arise as the direct sums of $h$ stable bundles (all of which automatically must be of rank $n'$ and degree $d'$). As we saw in Example \ref{ex:g=1ss}, this component may be identified with the symmetric product $\Sym^h\Gamma_X$. We note that the tropicalization map naturally lands in  $M_{n,d}^\oplus(\Gamma_X)$.

\begin{theorem}\label{thm_Tatecurve}
Let $X$ be a Tate curve and $n\geq 1$ as well as $d\in\Z$. Set $h=(n,d)$. Then the diagram
\begin{equation*}
    \begin{tikzcd}
         M_{n,d}(X)^{\an} \arrow[d,"\trop"'] \arrow[r,"\sim"] & \Sym^h X^{\an} \arrow[d,"\trop"]\\ 
        M_{n,d}^\oplus(\Gamma_X) \arrow[r,"\sim"]  & \Sym^h\Gamma_X \ ,
    \end{tikzcd}
\end{equation*}
commutes.
\end{theorem}

\begin{proof}
By definition, the decompositions into direct sums of stable bundles in $M_{n,d}(X)$ and $M_{n,d}^\oplus(\Gamma_X)$ are compatible with tropicalization. The deformation retraction to the skeleton is naturally compatible with the formation of symmetric powers (see~\cite{BrandtUlirsch, BrownMazzon, Shen_Lefschetz} for details). Thus we may reduce our statement to the case $h=1$ and we only need to prove the commutativity of the diagram
\begin{equation*}
    \begin{tikzcd}
         M_{n,d}(X)^{\an} \arrow[d,"\trop"'] \arrow[r,"\sim"] & X^{\an} \arrow[d,"\trop"]\\ 
        M_{n,d}^\oplus(\Gamma_X) \arrow[r,"\sim"]  & \Gamma_X \ ,
    \end{tikzcd}
\end{equation*}
when $(n,d)=1$. We note that the two horizontal identifications factor as 
\begin{equation*}
    \begin{tikzcd}
         M_{n,d}(X)^{\an} \arrow[d,"\trop"'] \arrow[r,"\sim"] & \Pic_d(X)^{\an} \arrow[d,"\trop"'] \arrow[r,"\sim"]& X^{\an} \arrow[d,"\trop"]\\ 
        M_{n,d}^\oplus(\Gamma_X) \arrow[r,"\sim"]  & \Pic_d(\Gamma_X) \arrow[r,"\sim"] & \Gamma_X \ ,
    \end{tikzcd}
\end{equation*}
where the two horizontal arrows on the left are given by $E\mapsto \det (E)$ (by~\cite[Cor.\ to Thm.\ 7]{Atiyah_ellipticvectorbundles} and Section~\ref{section_Atiyah}) and the horizontal arrows on the right are induced by the Abel--Jacobi theorems and tensoring with the base point. The diagram on the left commutes by Proposition~\ref{prop_tropbasicoperations} and the one on the right by~\cite[Theorem 1.3]{BakerRabinoff}.
\end{proof}

Theorem~\ref{thm_Tatecurve} together with the compatibility of skeletons and tropicalizations with symmetric powers that is treated in~\cite[Theorem A]{BrandtUlirsch} (also see ~\cite[Section 6]{Shen_Lefschetz}) implies Theorem~\ref{mainthm_Tatecurve} from the introduction. This also shows that the tropicalization map \begin{equation*}
    \trop\colon M_{n,d}(X)^{\an}\longrightarrow M_{n,d}^\oplus(\Gamma_X)
\end{equation*}
is continuous, surjective onto $M^\oplus_{n,d}(\Gamma_X)$, and proper. 

For a closed subscheme $Y\subseteq M_{n,d}(X)$ we set \begin{equation*}
    \Trop(Y):=\trop(Y^{\an})\subseteq M_{n,d}^\oplus(\Gamma) \ .
\end{equation*}
By a version of the classical Bieri--Groves theorem \cite[Theorem C]{BrandtUlirsch} the set $\Trop(Y)$ is the support of a $\val(K)$-rational polyhedral complex in $M_{n,d}^\oplus(\Gamma_X)\simeq \Sym^h(\Gamma_X)$.

 \subsection{A refined specialization inequality}\label{section_ellipticspecialization}

For semistable vector bundles on an elliptic curve we can make the specialization inequality from Proposition~\ref{prop_specializationinequality} more precise.

\begin{proposition}\label{prop_ellipticspecialization}
Let $E$ be a semistable vector bundle of rank $n$ and degree $d$ on a Tate curve $X$. Then we have the inequality
\begin{equation*}
    h^0(X, E)\leq r_{\Gamma_X}\big(\trop(E)\big)+1   \ .
\end{equation*}
Moreover, this is an equality if and only if $\deg(E)\neq 0$, or $\deg(E)=0$ and $E$ is a direct sum of a trivial bundle and a semistable bundle whose Jordan--H\"older summands are all line bundles with non-trivial tropicalization.
\end{proposition}

\begin{proof}
All components of the covering defining $\trop(E)$ are tropical elliptic curves, so a divisor $D$ with non-zero degree on any of them has rank $\max\{-1, \deg(D)-1\}$. As each component of $\trop(E)$ has the same slope as $E$, it follows that in the case where $\deg(E)<0$ we have 
\[
r_{\Gamma_X}\big(\trop(E)\big)+1 = k( -1 +1)=0 =h^0(X,E) \ ,
\]
where $k$ is the number of Jordan--H\"older factors of $E$ (it is shown in~\cite{Tu_semistablebundles} that $k=(n,d)$). If $\deg(E)>0$, then we have
\[
r_{\Gamma_X}\big(\trop(E)\big)+1= \sum_{i=1}^{k} (d_i-1+1)  = d=h^0(X,E) \ ,
\]
where the $d_i$ are the degrees of the components of $\trop(E)$ and the last equality follows from~\cite[Lemma 17]{Tu_semistablebundles}. 
Now assume that $\deg(E)=0$. Let $E=E_1\oplus\cdots \oplus E_k$ be a decomposition of $E$ into indecomposables. We have $h^0(X,E)= \sum h^0(X,E_i)$ as well as $r_{\Gamma_X}\big(\trop(E)\big)+1=\sum (r_{\Gamma_X}\big(\trop(E_i)\big)+1)$, and the multiset of Jordan--H\"older factors of $E$ is the union of the multisets of Jordan--H\"older factors of the $E_i$. Therefore, we may assume that $E=E_1$ is indecomposable. The Jordan--H\"older factors of $E$ are all line bundles of degree $0$, so $r_{\Gamma_X}\big(\trop(E)\big)+1$ is the number of Jordan--H\"older factors of $E$ with trivial tropicalization.  In particular, if $h^0(X,E)=0$ (in which case the inequality is trivially true) we have equality if and only if all Jordan--H\"older factors of $E$ have non-trivial tropicalization. On the other hand, if $h^0(X,E)\neq 0$, then $E\cong F_n$. The Jordan--H\"older factors of $F_n$ are all trivial, hence $r_{\Gamma_X}\big(\trop(E)\big)+1=n$. But $h^0(X,F_n)=1$ by~\cite[Lemma 15]{Atiyah_ellipticvectorbundles}, implying that the inequality holds, with equality if and only if $n=1$ and $E=F_1=\mc O_X$ is trivial. 
\end{proof}

\subsection{Brill--Noether loci}\label{section_BrillNoether}
On an elliptic curve $X$ all semistable vector bundle of degree $d>0$ are non-special. In fact, we have that $h^0(E^\ast\otimes \omega_X)=h^0(E^\ast)=0$, since $\deg E^\ast=-d<0$ and thus by the Weil--Riemann--Roch formula $h^0(X,E)=d$. Thus, the only interesting Brill--Noether loci are in degree $d=0$. Following~\cite{Tu_semistablebundles}, we therefore consider for $r\geq -1$ the Brill--Noether locus
\begin{equation*}
    W_{n,0}^r(X):=\big\{[E]\in M_{n,0}(X)\ \big\vert\ h^0(X,E')\geq r+1 \textrm{ for some } E'\sim E\big\} 
\end{equation*}
in $M_{n,0}(X)$. By~\cite[Lemma 19]{Tu_semistablebundles} the condition $h^0(X,E')\geq r+1$ for some $E'\sim E$ simplifies to $h^0\big(X,\gr(E)\big)\geq r+1$, and this is the case if and only at least $r+1$ summands are trivial. So we find~\cite[Theorem 4]{Tu_semistablebundles} stating that
\begin{equation*}
    W_{n,0}^r(X) \simeq
        \Sym^{n-r-1}(X) \ .
\end{equation*}

 For $r\geq -1$ we consider a tropical analogue of $W_{n,0}^r(X)$ by setting
\begin{equation*}
  W_{n,0}^r(\Gamma_X)=\big\{ E\in M_{n,0}^{\oplus}(\Gamma_X)\  \big\vert\ r_{\Gamma_X}(E)\geq r\big\} \ .
\end{equation*}
An element $E$ in $M_{n,0}^{\oplus}(\Gamma_X)$ is a direct sum of line bundles of degree zero, and, in order for $r_{\Gamma_X}(E)\geq r$ to hold, at least $r+1$ of them need to be trivial. Therefore we also have
\begin{equation*}
    W_{n,0}^r(\Gamma_X) \simeq 
        \Sym^{n-r-1}(\Gamma_X) \ .
\end{equation*}

\begin{proposition}
Let $X$ be a Tate curve and write $\Gamma_X$ for its minimal skeleton. Then we have 
\begin{equation*}
    \Trop\big(W_{n,0}^r(X)\big)=W_{n,0}^r(\Gamma_X) \ .
\end{equation*}
\end{proposition}

\begin{proof}
The inclusion $\subseteq$ is an immediate consequence of the specialization inequality coming from Proposition~\ref{prop_ellipticspecialization}. For the reverse inclusion $\supseteq$ we consider $E\in W_{n,0}^r(\Gamma_X)$. Then $E$ is a direct sum of $n$ line bundles of degree $0$, at least $r+1$ of which are trivial. So we may lift $E$ to the corresponding direct sum $F$ of line bundles on $X_{K'}$ (over a suitable extension $K'$ of $K$), where we lift each trivial line bundle to a trivial line bundle on $X_{K'}$ and each non-trivial line bundle to an automatically non-trivial line bundle. Then we have $h^0(X,F)-1=r_{\Gamma_X}(E)\geq r$ and thus $F\in W^r_{n,0}(X)(K')$ with $\trop(F)=E$. 
\end{proof}

\subsection{Generalized $\Theta$-divisors}\label{section_generalizedTheta}

Fix a stable vector bundle $F$ of rank $n'=\frac{n}{(n,d)}$ and degree $-d'=-\frac{d}{(n,d)}$. Then a natural generalized $\Theta$-divisor in $M_{n,d}(X)$ is defined by 
\begin{equation*}
    \Theta_F:=\big\{E\in M_{n,d}(X) \ \big\vert\  h^0(E\otimes F)\neq 0\big\} \ ,
\end{equation*}
which is well-defined by~\cite[Lemma 18 (i) and 20]{Tu_semistablebundles}. Then, by~\cite[Theorem 6]{Tu_semistablebundles}, under a suitable isomorphism $M_{n,d}(X)\simeq \Sym^h X$, we have 
\begin{equation*}
    \Theta_F=\Sym^{h-1}X\subseteq\Sym^h X \ ,
\end{equation*}
where one summand is fixed.

Denote by $F^{\trop}$ the tropicalization of $F$. A natural tropical analogue of the generalized $\Theta$-divisor is given by the locus
\begin{equation*}
    \Theta_{F^{\trop}}=\big\{E\in M_{n,d}^{\oplus}(\Gamma_X) \ \big\vert\   r_{\Gamma_X}(E\otimes F^{\trop})\neq -1\big\} \ .
\end{equation*}

\begin{proposition}
\label{prop_rank of tensor product}
Let $E\in M_{n,d}^\oplus(\Gamma_X)$. Then we have $r_{\Gamma_X}(E\otimes F^\trop)\geq 0$ if and only if $(F^\trop)^*$ is a direct summand of $E$. In particular, we have 
\begin{equation*}
    \Theta_{F^{\trop}}\simeq \Sym^{h-1}(\Gamma_X) \ .
\end{equation*}

\end{proposition}

\begin{proof}
It follows from Proposition~\ref{prop_multidivisors} that tensor products distribute over sums. 
Since the quantity $r_{\Ga_X}(\cdot)+1$ is additive with respect to direct sums, we may assume that $E$ is stable of rank $n'$ and degree $d'$. In this case,  the bundles $E$ and $F^\trop$ are represented divisors $D$ and $D'$ of degree $d'$ and $-d'$, respectively, on the domain of the unique degree $n'$ cover $\widetilde \Gamma\to \Gamma_X$. By Proposition~\ref{prop_multidivisors}, the bundle $E\otimes F$ is represented by the divisor $D\boxtimes D'$ on $\widetilde \Gamma \times_{\Gamma_X} \widetilde\Gamma$. If $\sigma_1,\ldots, \sigma_{n'}$ denote the automorphisms of the cover, we obtain a homeomorphism 
\[
\bigsqcup_{i=1}^{n'} \widetilde\Gamma \to \widetilde \Gamma\times_{\Gamma_X} \widetilde\Gamma
\]
given by $\sigma_i\times \id$ on the $i$-th copy of $\widetilde\Gamma$. With this identification, $D\boxtimes D'$ is given by $D+ \sigma_i^*(D')$ on the $i$-th copy of $\widetilde\Gamma$. We have $r_{\Gamma_X}(E\otimes F^\trop)\geq 0$ if and only if $r_{\widetilde \Gamma\times_{\Gamma_X} \widetilde\Gamma}(D+\sigma_i^*(D'))\geq 0$ for some $i$. Because each $D+\sigma_i^*(D')$ has degree zero, this is the case if and only if $D\sim \sigma_i^*(-D')$ for some $i$. This in turn happens if and only if $D$ and $-D'$ define the same bundle. As $-D'$ defines $(F^\trop)^*$ by Proposition~\ref{prop_multidivisors}, this happens if and only if $E= (F^\trop)^*$. 

To show that $\Theta_{F^{\trop}}\simeq \Sym^{h-1}(\Gamma_X)$, we observe that $M_{n,d}^\oplus(\Gamma_X)$ is identified with $\Sym^h(\Gamma_X)$ by decomposing a semistable bundle into its summands and then identifying a stable bundle with a point on $\Gamma_X$. By what we have just proved, $\Theta_{F^\trop}$ is precisely the locus in $M_{n,d}^\oplus(\Gamma_X)$ where one summand is equal to $(F^\trop)^*$. This fixes one of the $h$ points of $\Gamma_X$ corresponding to a bundle in $\Theta_{F^\trop}$, allowing us to identify $\Theta_{F^\trop}$ with $\Sym^{h-1}(\Gamma_X)$
\end{proof}

\begin{remark}
A caveat: The vector bundle $E\otimes F^\trop$ in the proof of Proposition~\ref{prop_rank of tensor product} above is not contained in $M^\oplus_{n\cdot n',0}(\Gamma_X)$. In particular, for $E\in M_{n,d}(X)$ we have $\trop(E\otimes F)\neq \trop(E)\otimes F^\trop$. Thus, tensor products do not commute with the tropicalization on the moduli space of semistable vector bundles on a Tate curve. This is in stark contrast to Proposition~\ref{prop_tropbasicoperations}.
\end{remark}

\begin{proposition}
Let $X$ be a Tate curve and write $\Gamma_X$ for its minimal skeleton. Then we have 
\begin{equation*}
    \Trop(\Theta_F)=\Theta_{F^{\trop}} \ .
\end{equation*}
\end{proposition}

\begin{proof}
As shown in the proof of ~\cite[Theorem 6]{Tu_semistablebundles}, the theta divisor $\Theta_F$ consists precisely of those semistable bundles of rank $n$ and degree $d$ that have a Jordan--H\"older factor equal to $F^*$. This shows that $\Trop(\Theta_F)\subseteq \Theta_{F^\trop}$. For the converse, let $E\in \Theta_{F^{\trop}}$. By Proposition~\ref{prop_rank of tensor product}, one summand of $E$ is equal to $(F^\trop)^*$; we denote by $E_2,\ldots, E_h$ the remaining summands. By Theorem~\ref{thm_Tatecurve}, for each $i=2,\ldots,h$ there exists a stable vector bundle $G_i$ on $X$ of rank $n'$ and degree $d'$ with $\trop(G_i)=E_i$. Let $G=F^*\oplus \bigoplus_{i=2}^h G_i$. By construction, we have $G\in \Theta_F$ and $\trop(G)=E$.
\end{proof}



 \bibliographystyle{amsalpha}
\bibliography{biblio}{}

\end{document}